\documentclass{aptpub}
\usepackage{a4wide,enumerate}
\usepackage[colorlinks,citecolor=blue]{hyperref}

\authornames{Rafa{\l} Kulik and Philippe Soulier}
\shorttitle{Stochastic volatility models with long memory and infinite  variance}

\begin{document}

\title{Limit theorems for long memory \\ stochastic volatility models with infinite\\
  variance: Partial Sums and Sample Covariances}

\authorone[University of Ottawa]{Rafa{\l} Kulik}

\authortwo[Universit\'e Paris Ouest-Nanterre]{Philippe  Soulier}

\addressone{Corresponding author: Department of Mathematics and Statistics,
  University of Ottawa, 585 King Edward Avenue, Ottawa ON K1N 6N5, Canada,
  email: rkulik@uottawa.ca}

\addresstwo{D\'epartement de Math\'ematiques, Universit\'e Paris Ouest-Nanterre,
  200, Avenue de la R\'epublique 92000, Nanterre Cedex, France, email:
  philippe.soulier@u-paris10.fr}
\begin{abstract}
  In this paper we extend the existing literature on the asymptotic behaviour of
  the partial sums and the sample covariances of long memory stochastic
  volatility models in the case of infinite variance. We also consider models
  with leverage, for which our results are entirely new in the infinite variance
  case. Depending on the interplay between the tail behaviour and the intensity
  of dependence, two types of convergence rates and limiting distributions can
  arise. In particular, we show that the asymptotic behaviour of partial sums is
  the same for both LMSV and models with leverage, whereas there is a crucial
  difference when sample covariances are considered.
\end{abstract}



\section{Introduction}
One of the standardized features of financial data is that returns
are uncorrelated, but their squares, or absolute values, are
(highly) correlated, a property referred to as long memory (which
will be later defined precisely). A second commonly accepted feature
is that log-returns are heavy tailed, in the sense that some moment
of the log-returns is infinite. The last one we want to mention is
leverage. In the financial time series context, leverage is
understood to mean negative dependence between previous returns and
future volatility (i.e. a large negative return will be followed by
a high volatility). Motivated by these empirical findings, one of
the common modeling approaches is to represent log-returns $\{Y_i\}$
as a stochastic volatility sequence $Y_i = Z_i \sigma_i$ where
$\{Z_i\}$ is an i.i.d.~sequence and $\{\sigma_i^2\}$ is the
conditional variance or more generally a certain process which
stands as a proxy for the volatility. In such a process, long memory
can only be modeled through the sequence $\{\sigma_i\}$, and the
tails can be modeled either through the sequence $\{Z_i\}$ or
through $\{\sigma_i\}$, or both. The well known GARCH processes
belong to this class of models. The volatility sequence
$\{\sigma_i\}$ is heavy tailed, unless the distribution of $Z_0$ has
finite support, and leverage can be present. But long memory in
squares cannot be modeled by GARCH process. The FIGARCH process was
introduced by \cite{baillie:bollerslev:mikkelsen:1996} to this
purpose, but it is not known if it really has a long memory
property, see e.g. \cite{douc:roueff:soulier:2008}.

To model long memory in squares, the so-called \textit{Long Memory in Stochastic
  Volatility} (LMSV) process was introduced in \cite{breidt:crato:delima:1998},
generalizing earlier short memory version of this model. In this model, the
sequences $\{Z_i\}$ and $\{\sigma_i\}$ are fully independent, and $\{\sigma_i\}$
is the exponential of a Gaussian long memory process. Tails and long memory are
easily modeled in this way, but leverage is absent. Throughout the paper, we
will refer to this process as LMSV, even though we do not rule out the short
memory case.

In order to model leverage, \cite{Nelson1991} introduced the EGARCH model (where
E stands for exponential), later extended by \cite{bollerslev:mikkelsen:1996} to
the FIEGARCH model (where FI stands for fractionally integrated) in order to
model also long memory. In these models, $\{Z_i\}$ is a Gaussian white noise,
and $\{\sigma_i\}$ is the exponential of a linear process with respect to a
function of the Gaussian sequence $\{Z_i\}$.  \cite{SurgailisViano2002} extended
the type of dependence between the sequences $\{Z_i\}$ and $\{X_i\}$ and relaxed
the Gaussian assumption for both sequences, but assumed finite moments of all
order. Thus long memory and leverage are possibly present in these models, but
heavy tails are excluded.

A quantity of other models have been introduced, e.g.  models of
Robinson and Zaffaroni \cite{RobinsonZaffaroni1997},
\cite{RobinsonZaffaroni1998} and their further extensions in
\cite{Robinson2001}; LARCH($\infty$) processes
\cite{GiraitisRobinsonSurgailis2000} and their bilinear extensions
\cite{GiraitisSurgailis2002}, and LARCH$_{+}(\infty)$
\cite{Surgailis2008}; to mention a few. All of these models have
long memory and some have leverage and allow for heavy tails.  The
theory for these models is usually extremely involved, and only the
asymptotic properties of partial sums are known in certain
cases. We will not consider these models here.
In~\cite{giraitis:leipus:robinson:surgailis:2004} the leverage
effect and long memory property of a LARCH($\infty$) model was
studied thoroughly.

The theoretical effect of long memory is that the covariance of absolute powers
of the returns $\{Y_i\}$ is slowly decaying and non summable. This induces non
standard limit theorems, such as convergence of the partial sum process to the
fractional Brownian motion or finite variance non Gaussian processes or even
L\'evy processes. In practice, long memory is often evidenced by sample
covariance plots, showing an apparent slow decay of the covariance function.
Therefore, it is of interest to investigate the asymptotic behaviour of the
sample mean or of the partial sum process, and of the sample variance and
covariances.

In the case where $\sigma_i=\sigma(X_i)$, $\{X_i\}$ is a stationary
Gaussian process with summable covariances and $\sigma(x) =
\exp(x)$, the asymptotic theory for sample mean of LMSV processes
with infinite variance is a straightforward consequence of a point
process convergence result in \cite{DavisMikosch2001}. The limit is
a L\'evy stable process. \cite{SurgailisViano2002} considered the
convergence of the partial sum process of absolute powers of
generalized EGARCH processes with finite moments of all orders and
showed convergence to the fractional Brownian motion. To the best of
our knowledge, the partial sum process of absolute powers has never
been studied in the context of heavy tails and long memory and
possible leverage, for a general function $\sigma$.

The asymptotic theory for sample covariances of weakly dependent
stationary processes with finite moments dates back to Anderson, see
\cite{anderson:1971}. The case of linear processes with regularly
varying innovations was studied in \cite{davis:resnick:1985m} and
\cite{DavisResnick1986}, for infinite variance innovation and for
innovations with finite variance but infinite fourth moment,
respectively. The limiting distribution of the sample covariances
(suitably centered and normalized) is then a stable law. These
results were obtained under conditions that rule out long memory.
For infinite variance innovation with tail index $\alpha\in(1,2)$,
these results were extended to long memory linear processes by
\cite{kokoszka:taqqu:1996}. The limiting distributions of the sample
covariances are again stable laws. However, if $\alpha \in (2,4)$,
\cite{horvath:kokoszka:2008} showed that as for partial sums, a
dichotomy appears: the limiting distribution and the rate of
convergence depend on an interplay between a memory parameter and
the tail index $\alpha$. The limit is either stable (as in the
weakly dependent or i.i.d.~case) or, if the memory is strong enough,
the limiting distribution is non Gaussian but with finite variance
(the so-called Hermite-Rosenblatt distributions).  If the fourth
moment is finite, then the dichotomy is between Gaussian or finite
variance non Gaussian distributions (again of Hermite-Rosenblatt
type); see \cite{hosking:1996}, \cite[Theorem
3.3]{horvath:kokoszka:2008} and \cite{wu:huang:zheng:2010}.

The asymptotic properties of sample autocovariances of GARCH processes have been
studied by \cite{basrak:davis:mikosch:2002r}. Stable limits arise as soon as the
marginal distribution has an infinite fourth moment. \cite{DavisMikosch2001}
studied the sample covariance of a zero mean stochastic volatility process,
under implicit conditions that rule out long memory, and also found stable
limits. \cite{mcelroy:politis:2007} (generalized by
\cite{jach:mcelroy:politis:2011}) studied partial sums and sample variance of a
possibly nonzero mean stochastic volatility process with infinite variance and
where the volatility is a Gaussian long memory process (in which case it is not
positive but this is not important for the theoretical results). They obtained a
dichotomy between stable and finite variance non Gaussian limits, and also the
surprising result that when the sample mean has a long memory type limit, then
the studentized sample mean converges in probability to zero.

The first aim of this article is to study asymptotic properties of partial sums,
sample variance and covariances of stochastic volatility processes where the
volatility is an arbitrary function of a Gaussian, possibly long memory process
$\{X_i\}$ independent of the sequence $\{Z_i\}$, which is a heavy tailed
i.i.d.~sequence. We refer to these processes as LMSV processes. The interest of
considering other functions than the exponential function is that it allows to
have other distributions than the log-normal for the volatility, while keeping
the convenience of Gaussian processes, without which dealing with long memory
processes becomes rapidly extremely involved or even intractable. The results we
obtain extend in various aspects all the previous literature in this domain.

Another important aim of the paper is to consider models with possible
leverage. To do this, we need to give precise assumptions on the nature of the
dependence between the sequences $\{Z_i\}$ and $\{X_i\}$, and since they are
related in the process $\{Y_i\}$ through the function $\sigma$, these
assumptions also involve the function $\sigma$. We have not looked for the
widest generality, but the functions $\sigma$ that we consider include the
exponential functions and all symmetric polynomials with positive coefficients.
This is not a severe restriction since the function $\sigma$ must be
nonnegative. Whereas the asymptotic theory for the partial sums is entirely
similar to the case of LMSV process without leverage, asymptotic properties of
sample autocovariances may be very different in the presence of leverage. Due to
the dependence between the two sequences, the rates of convergence and
asymptotic distribution may be entirely different when not stable.

The article is organized as follows. In Section \ref{sec:prel} we formulate
proper assumptions, as well as prove some preliminary results on the marginal
and multivariate tail behaviour of the sequence $\{Y_i\}$.  In
Section~\ref{sec:pp-conv}, we establish the limit theory for a point process
based on the rescaled sequence $\{Y_i\}$. This methodology was first used in
this context by \cite{DavisMikosch2001} and our proofs are closely related to
those in this reference. Section \ref{sec:partial-sums} applies these results to
obtain the functional asymptotic behaviour of the partial sum process of the
sequences $\{Y_i\}$ and of powers. In Section \ref{sec:sample-covariances} the
limiting behaviour of the sample covariances and autocorrelation of the process
$\{Y_i\}$ and of its powers is investigated. Proofs are given in Section
\ref{sec:proofs}. In the Appendix we recall some results on multivariate
Gaussian processes with long memory.

\subsection*{A note on the terminology}
We consider in this paper sequences $\{Y_i\}$ which can be expressed as $Y_i=Z_i
\sigma(X_i)=Z_i\sigma_i$, where $\{Z_i\}$ is an i.i.d.  sequence and $Z_i$ is
independent of $X_i$ for each $i$. Originally, SV and LMSV processes refer to
processes where the sequences $\{Z_i\}$ and $\{\sigma_i\}$ are fully
independent, $\sigma_i=\sigma(X_i)$, $\{X_i\}$ is a Gaussian process and
$\sigma(x)=\exp(x)$; see e.g.  \cite{breidt:crato:delima:1998},
\cite{BreidtDavis1998}, \cite{DavisMikosch2001}. The names EGARCH and FIEGARCH,
introduced respectively by \cite{Nelson1991} and
\cite{bollerslev:mikkelsen:1996}, refer to the case where $\sigma(x)=\exp(x)$
and where $\{X_i\}$ is a non Gaussian process which admits a linear
representation with respect to an instantaneous function of the Gaussian
i.i.d.~sequence $\{Z_i\}$, with dependence between the sequences $\{Z_i\}$ and
$\{X_i\}$.  \cite{SurgailisViano2002} still consider the case
$\sigma(x)=\exp(x)$, but relax the assumptions on $\{Z_i\}$ and $\{X_i\}$, and
retain the name EGARCH. The LMSV processes can be seen as border cases of EGARCH
type processes, where the dependence between the sequences $\{Z_i\}$ and
$\{X_i\}$ vanishes.  In this article, we consider both LMSV models, and models
with leverage which generalize the EGARCH models as defined by
\cite{SurgailisViano2002}.  In order to refer to the latter models, we have
chosen not to use the acronym EGARCH or FIEGARCH, since these models were
defined with very precise specifications and this could create some confusion,
nor to create a new one such as GEGARCH (with $G$ standing twice for
generalized, which seems a bit too much) or (IV)LMSVwL (for (possibly) Infinite
Variance Long Memory Stochastic Volatility with Leverage). Considering that the
main feature which distinguishes these two classes of models is the presence or
absence of leverage, we decided to refer to LMSV models when leverage is
excluded, and to models with leverage when we include the possibility thereof.

\section{Model description, assumptions and tail behaviour}
\label{sec:prel}

Let $\{Z_i,i\in\mathbb{Z}\}$ be an i.i.d.~sequence whose marginal
distribution has  regularly varying tails:
\begin{equation}
  \label{eq:model-2}
  \lim_{x\to+\infty} \frac{\mathbb{P}(Z_0>x)}{x^{-\alpha}L(x)} = \beta \; , \quad  \lim_{x\to+\infty}
  \frac{\mathbb{P}(Z_0<-x)}{x^{-\alpha}L(x)} = 1-\beta \; ,
\end{equation}
where $\alpha>0$, $L$ is slowly varying at infinity, and $\beta\in
[0,1]$. Condition~(\ref{eq:model-2}) is referred to as the Balanced Tail
Condition.  It is equivalent to assuming that $\mathbb{P}(|Z_0|>x) = x^{-\alpha} L(x)$
and
\begin{align*}
  \beta = \lim_{x\to+\infty} \frac{\mathbb{P}(Z_0>x)}{\mathbb{P}(|Z_0|>x)} = 1 -
  \lim_{x\to+\infty} \frac{\mathbb{P}(Z_0<-x)}{\mathbb{P}(|Z_0|>x)} \; .
\end{align*}
We will say that two random variables $Y$ and $Z$ are right-tail equivalent if
there exists $c\in(0,\infty)$ such that
\begin{align*}
  \lim_{x\to+\infty} \frac{\mathbb{P}(Y>x)}{\mathbb{P}(Z>x)} = c  \; .
\end{align*}
If one of the random variables has a regularly varying right tail,
then so has the other, with the same tail index. The converse is
false, i.e. two random variables can have the same tail index
without being tail equivalent. Two random variables $Y$ and $Z$ are
said to be left-tail equivalent if $-Y$ and $-Z$ are right-tail
equivalent, and they are said to be tail equivalent if they are both
left- and right-tail equivalent.

Under (\ref{eq:model-2}), if moreover $\mathbb{E}\left[|Z_0|^{\alpha}\right] =
\infty$, then $Z_1Z_2$ is regularly varying and (see e.g.
\cite[Equation~(3.5)]{DavisResnick1986})
\begin{align*}
  \lim_{x\to+\infty} \frac{\mathbb{P}(Z_0>x)} {\mathbb{P}(Z_0Z_1>x)} & = 0 \; , \\
  \lim_{x\to+\infty} \frac{\mathbb{P}(Z_1Z_2 > x)} {\mathbb{P}(|Z_1Z_2| > x)} & = \beta^2 +
  (1-\beta)^2 \; .
\end{align*}
For example, if (\ref{eq:model-2}) holds and the tail of $|Z_0|$ has
  Pareto-type tails, i.e.  $\mathbb{P}(|Z_0|>x) \sim cx^{-\alpha}$ as $x\to+\infty$ for
  some $c>0$, then $\mathbb{E}\left[|Z_0|^{\alpha}\right] = \infty$.  We will further
assume that $\{X_i\}$ is a stationary zero mean unit variance Gaussian process
which admits a linear representation with respect to an i.i.d.~Gaussian white
noise $\{\eta_i\}$ with zero mean and unit variance, i.e.
\begin{align}
  \label{eq:linear}
  X_i = \sum_{j=1}^{\infty} c_j \eta_{i-j}
\end{align}
with $\sum_{j=1}^\infty c_j^2=1$. We assume that the process $\{X_i\}$ either
has short memory, in the sense that its covariance function is absolutely
summable, or exhibits long memory with Hurst index $H \in (1/2,1)$, i.e. its
covariance function $\{\rho_n\}$ satisfies
\begin{equation}
  \label{eq:model-1}
  \rho_n =   \mathrm{cov}(X_0,X_n) = \sum_{j=1}^\infty c_jc_{j+n} = n^{2H-2} \ell(n) \; ,
\end{equation}
where $\ell$ is a slowly varying function.

Let $\sigma$ be a deterministic, nonnegative and continuous function
defined on $\mathbb R$.  Define $\sigma_i = \sigma(X_i)$ and the
stochastic volatility process $\{Y_i\}$ by
\begin{equation}
  \label{eq:model-3}
  Y_i = \sigma_i Z_i = \sigma(X_i) Z_i  \; .
\end{equation}
At this moment we do not assume independence of $\{\eta_i\}$ and $\{Z_i\}$.  Two
special cases which we are going to deal with are:
\begin{itemize}
\item Long Memory Stochastic Volatility (LMSV) model: where $\{\eta_i\}$ and
  $\{Z_i\}$ are independent.
\item Model with leverage: where $\{(\eta_i,Z_i)\}$ is a sequence of
  i.i.d.~random vectors. For fixed $i$, $Z_i$ and $X_i$ are independent, but
  $X_{i}$ may not be independent of the past $\{Z_j, j<i\}$.
\end{itemize}

Both cases are encompassed in the following assumption which will be in force
throughout the paper.
\begin{assumption}
  \label{hypo:iid-bivarie}
The Stochastic Volatility process $\{Y_i\}$ is defined by
\begin{align*}
  Y_i = \sigma_i Z_i \; ,
\end{align*}
where $\sigma_i = \sigma(X_i)$, $\{X_i\}$ is a Gaussian linear
process with respect to the i.i.d.~sequence $\{\eta_i\}$ of standard
Gaussian random variables such that (\ref{eq:linear}) holds,
$\sigma$ is a nonnegative function such that
$\mathbb{P}(\sigma(a\eta_0)>0)=1$ for all $a\ne0$, $\{(Z_i,\eta_i)\}$ is an
i.i.d.~sequence and $Z_0$ satisfies the Balanced Tail
Condition~(\ref{eq:model-2}) with $\mathbb{E}[|Z_0|^\alpha]=\infty$.
\end{assumption}

Let $\mathcal F_i$ be the sigma-field generated by $\eta_j,Z_j$, $j\leq i$. Then
the following properties hold.
\begin{itemize}
\item $Z_i$ is $\mathcal F_i$-measurable and independent of $\mathcal F_{i-1}$;
\item $X_{i}$ and $\sigma_i$ are $\mathcal F_{i-1}$-measurable.
\end{itemize}

We will also impose the following condition on the continuous function
$\sigma$. There exists $q>0$ such that
\begin{gather}
  \label{eq:sigma-assumption}
  \sup_{0\le\gamma\le 1} \mathbb{E}\left[\sigma^{q}(\gamma X_0)\right] < \infty \; .
\end{gather}
It is clearly fulfilled for all $q,q'$ if $\sigma$ is a polynomial
or $\sigma(x) = \exp(x)$ and $X_0$ is a standard Gaussian random
variable.  Note that if (\ref{eq:sigma-assumption}) holds for some
$q>0$, then, for $q'\leq q/2$, it holds that
\begin{align*}
  \sup_{0\le\gamma\le 1} \mathbb{E} \left[ \sigma^{q'}(\gamma X_0) \sigma^{q'}(\gamma
    X_{s}) \right] < \infty \; , \ s=1,2,\dots
\end{align*}

\subsection{Marginal tail behaviour}
If (\ref{eq:sigma-assumption}) holds, then clearly $\mathbb{E}[\sigma^{q}(X_0)] <
\infty$.  If moreover $q>\alpha$, since $X_i$ and $Z_i$ are independent for
fixed $i$, Breiman's Lemma (see e.g.  \cite[Proposition 7.5]{resnick:2007})
yields that the distribution of $Y_0$ is regularly varying and
\begin{equation}
  \label{eq:Breiman-1}
  \lim_{x\to+\infty}
  \frac{\mathbb{P}(Y_0>x)}{\mathbb{P}(Z_0>x)} =  \lim_{x\to+\infty}
  \frac{\mathbb{P}(Y_0 < -x)}{\mathbb{P}(Z_0 < -x)} = \mathbb{E}[\sigma^{\alpha}(X_0)] \; .
\end{equation}
Thus we see that there is no effect of leverage on marginal tails.
Define
\begin{align}
  \label{eq:def-an}
  a_n = \inf\{x: \mathbb{P}(|Y_0|>x) < 1/n\} \; .
\end{align}
Then the sequence $a_n$ is regularly varying at infinity with index $1/\alpha$.
Moreover, since $\sigma$ is nonnegative, $Z_0$ and $Y_0$ have the same skewness,
i.e.
\begin{align*}
  & \lim_{n\to+\infty} n \mathbb{P}(Y_0>a_n) = 1 - \lim_{n\to+\infty} n \mathbb{P}(Y_0<-a_n) =
  \beta \; .
\end{align*}

\subsection{Joint exceedances}
One of the properties of heavy tailed stochastic volatility models is that
large values do not cluster. Mathematically, for all $h>0$,
\begin{equation}
  \label{eq:Breiman-0b}
  \mathbb{P}(|Y_0| > x,|Y_h| > x) = o(\mathbb{P}(|Y_0|>x)) \; .
\end{equation}
For the LMSV model, conditioning on $\sigma_0,\sigma_h$ yields
\begin{equation}
  \label{eq:Breiman-0}
  \lim_{x\to+\infty}\frac{\mathbb{P}(|Y_0|>x,|Y_h|>x)}{\mathbb{P}^2(|Z_0|>x)}=\mathbb{E}[(\sigma_0\sigma_h)^{\alpha}] \; ,
\end{equation}
if~(\ref{eq:sigma-assumption}) holds for some $q>2\alpha$.  Property
(\ref{eq:Breiman-0b}) still holds when leverage is present.  Indeed, let $F_Z$
denote the distribution function of $Z_0$ and $\bar F_Z=1-F_Z$. Recall that
$\mathcal F_{h-1}$ is the sigma-field generated by $\eta_j,Z_j,j\le h-1$. Thus,
$Y_0$ and $X_h$ are measurable with respect to $\mathcal F_{h-1}$, and $Z_h$ is
independent of $\mathcal F_{h-1}$. Conditioning on $\mathcal F_{h-1}$ yields
\begin{align*}
  \mathbb{P}(Y_0>x, Y_h>x) = \mathbb{E}[ \bar F_Z(x/\sigma_h)
  \mathbf{1}_{\{Y_0>x\}}]\; .
\end{align*}
Next, fix some $\epsilon>0$. Applying Lemma~\ref{lem:bound-potter},
there
  exists a constant $C$ such that for all $x \ge 1$,
\begin{align*}
  \frac{\mathbb{P}\left(Y_0>x, Y_h>x\right)}{\mathbb{P}(Z_0>x)} = \mathbb{E} \left[ \frac{\bar
      F_Z(x/{\sigma_h})} {\bar F_Z(x)} \mathbf 1_{\{Y_0>x\}}\right]\le C
  \mathbb{E}\left[(1\vee \sigma_h)^{\alpha+\epsilon} \mathbf 1_{\{Y_0>x\}}\right].
\end{align*}
If (\ref{eq:sigma-assumption}) holds for some $q>\alpha$, and $\epsilon$ is
chosen small enough so that $\alpha+\epsilon<q$, then by bounded convergence,
the latter expression is finite and converges to 0 as $x\to+\infty$.

\subsection{Products}
For the LMSV model, another application of Breiman's Lemma yields that $Y_0Y_h$
is regularly varying for all $h$. If (\ref{eq:sigma-assumption}) holds for some
$q>2\alpha$, then
\begin{equation}
  \label{eq:Breiman-2}
  \lim_{x\to+\infty}
  \frac{\mathbb{P}(Y_0Y_h>x)}{\mathbb{P}(Z_0Z_1>x)}=\mathbb{E}[(\sigma_0\sigma_h)^{\alpha}] \; ,
  \quad \lim_{x\to+\infty}
  \frac{\mathbb{P}(Y_0Y_h<-x)}{\mathbb{P}(Z_0Z_1<-x)}=\mathbb{E}[(\sigma_0\sigma_h)^{\alpha}] \;.
\end{equation}
For further reference, we gather in a Lemma some properties of the
products in the LMSV case, some of which are mentioned in
\cite{DavisMikosch2001} in the case $\sigma(x)=\exp(x)$.
\begin{lem}
  \label{lem:asympt-indep-lmsv}
  Let Assumption~\ref{hypo:iid-bivarie} hold and let the sequences $\{\eta_i\}$
  and $\{Z_i\}$ be mutually independent.  Assume that
  (\ref{eq:sigma-assumption}) holds with $q>2\alpha$. Then $Y_0Y_1$ is tail
  equivalent to $Z_0Z_1$ and has regularly varying and balanced tails with index
  $\alpha$. Moreover, for all $h\geq1$, there exist real numbers $d_+(h)$,
  $d_-(h)$ such that
  \begin{align}
    \lim_{x\to \infty} \frac{\mathbb{P}(Y_0Y_h>x)}{\mathbb{P}(|Y_0Y_1|>x)} = d_+(h) \; , \ \
    \lim_{x\to \infty} \frac{\mathbb{P}(Y_0Y_h<-x)}{\mathbb{P}(|Y_0Y_1|>x)} = d_-(h) \;
    . \label{eq:defd+-}
  \end{align}
  Let $b_n$ be defined by
  \begin{align}
    \label{eq:def-bn}
    b_n = \inf\{x: \mathbb{P}(|Y_0Y_1|>x) \leq 1/n\} \; .
  \end{align}
  The sequence $\{b_n\}$ is regularly varying with index $1/\alpha$ and
  \begin{align}
    a_n = o(b_n) \; . \label{eq:domination}
  \end{align}
  For all $i\ne j>0$, it holds that
  \begin{gather}
    \lim_{n\to \infty} n\mathbb{P}(|Y_0| > a_n x \; , \ |Y_0Y_j| > b_n x) = 0 \; ,    \label{eq:indep1} \\
    \lim_{n\to \infty} n\mathbb{P}(|Y_0Y_i|> b_n x \; , \ |Y_0Y_j| > b_n x) = 0 \;    . \label{eq:indep-products}
\end{gather}

\end{lem}
The quantities $d_+(h)$ and $d_-(h)$ can be easily computed in the LMSV case.
\begin{align*}
  d_+(h) & = \{\beta^2+(1-\beta)^2\} \frac
  {\mathbb{E}[\sigma^\alpha(X_0)\sigma^\alpha(X_h)]}
  {\mathbb{E}[\sigma^\alpha(X_0)\sigma^\alpha(X_1)]} \; ,  \ \
  d_-(h)  = 2\beta(1-\beta) \frac {\mathbb{E}[\sigma^\alpha(X_0)\sigma^\alpha(X_h)]}
  {\mathbb{E}[\sigma^\alpha(X_0)\sigma^\alpha(X_1)]} \; .
\end{align*}

When leverage is present, many different situations can occur,
obviously depending on the type of dependence between $Z_0$ and
$\eta_0$, and also on the function~$\sigma$. We consider the
exponential function $\sigma(x)=\exp(x)$, and a class of subadditive
functions. In each case we give an assumption on the type of
dependence between $Z_0$ and $\eta_0$ that will allow to prove our
results. Examples are given after the Lemmas.

\begin{lem}
  \label{lemma:asymp-indep-EGARCH-expo}
  Assume that $\sigma(x) = \exp(x)$ and $\exp(k\eta_0) Z_0$ is tail
  equivalent to $Z_0$ for all $k \in\mathbb R$. Then~all the conclusions of
  Lemma~\ref{lem:asympt-indep-lmsv} hold.
\end{lem}

\begin{lem}
  \label{lemma:asymp-indep-EGARCH-quadratic}
  Assume that the function $\sigma$ is subadditive, i.e. there exists a constant
  $C>0$ such that for all $x,y\in\mathbb R$, $\sigma(x+y) \leq
  C\{\sigma(x)+\sigma(y)\}$. Assume that for any $a,b>0$,
  $\sigma(a\xi+b\eta_0)Z_0$ is tail equivalent to $Z_0$, where $\xi$ is a
  standard Gaussian random variable independent of $\eta_0$, and
  $\sigma(b\eta_0)Z_0$ is either tail equivalent to $Z_0$ or
  $\mathbb{E}[\{\sigma(b\eta_0)|Z_0|\}^q]<\infty$ for some $q>\alpha$. Then~all the
  conclusions of Lemma~\ref{lem:asympt-indep-lmsv} hold.
\end{lem}

\begin{example}
  \label{ex:decreasing}

  Assume that $Z_0 = |\eta_0|^{-1/\alpha}U_0$
  with $\alpha>0$, where $U_0$ is independent of $\eta_0$ and $\mathbb{E}[|U_0|^q]<\infty$
  for some $q>\alpha$. Then $Z_0$ is regularly varying with index~$-\alpha$.
  \begin{itemize}
  \item Case $\sigma(x)=\exp(x)$. For each $c>0$, $Z_0\exp(c\eta_0)$
    is tail equivalent to $Z_0$. See Lemma~\ref{lem:tail-equivalence} for a
    proof of this fact.
  \item Case $\sigma(x)=x^2$. Let $q' \in (\alpha,q \wedge
    \{\alpha/(1-2\alpha)_+\})$. Then
    \begin{align*}
      \mathbb{E}[\sigma^{q'}(b\eta_0)|Z_0|^{q'}] = b^{2q'} \mathbb{E}[|\eta_0|^{q'(2-1/\alpha)} |U_0|^{q'}] < \infty \; .
    \end{align*}
    Furthermore, let $\xi$ be a standard Gaussian random variable independent of
    $\eta_0$ and~$Z_0$. Then,
    \begin{align*}
      \sigma(a\xi+b\eta_0)Z_0 = a^2\xi^2 Z_0 + 2ab\xi
      \mathrm{sign}(\eta_0)|\eta_0|^{1-1/\alpha} U_0 + b^2 |\eta_0|^{2-1/\alpha}
      U_0 \; .
    \end{align*}
    Since $\xi$ is independent of $Z_0$ and Gaussian, by Breiman's lemma, the
    first term on the right-hand side of the previous equation is tail
    equivalent to $Z_0$. The last two terms have finite moments of order $q'$
    for some $q'>\alpha$ and do not contribute to the tail. Thus the
    assumptions of Lemma~\ref{lemma:asymp-indep-EGARCH-quadratic} are satisfied.
  \end{itemize}
\end{example}

\begin{example}
  \label{xmpl:indep}
 Let $Z_0'$ have regularly varying balanced tails with index $-\alpha$,
 independent of $\eta_0$. Let $\Psi_1(\cdot)$ and $\Psi_2(\cdot)$ be polynomials and
 define $Z_0= Z_0'\Psi_1(\eta_0)+\Psi_2(\eta_0)$. Then, by Breiman's Lemma,
 $Z_0$ is tail equivalent to $Z'_0$, and it is easily checked that the
 assumptions of Lemma~\ref{lemma:asymp-indep-EGARCH-expo} are satisfied and the
 assumptions of Lemma~\ref{lemma:asymp-indep-EGARCH-quadratic} are satisfied with
 $\sigma$ being any symmetric polynomial with positive coefficients. We omit the
 details.
\end{example}

\section{Point process convergence}
\label{sec:pp-conv}
For $s=0,\ldots,h$, define a Radon measure $\lambda_s$ on
$[-\infty,\infty]\setminus \{0\}$ by
\begin{align*}
  \lambda_{0}(\mathrm d x) = \alpha \left\{ \beta x^{-\alpha-1} \mathbf
    1_{(0,\infty)}(x) + (1-\beta) (-x)^{-\alpha-1} \mathbf
    1_{(-\infty,0)}(x)  \right\} \mathrm d x \; , \\
  \lambda_{s}(\mathrm d x) = \alpha \left\{ d_+(s) x^{-\alpha-1} \mathbf
    1_{(0,\infty)}(x) + d_-(s) (-x)^{-\alpha-1} \mathbf 1_{(-\infty,0)}(x)
  \right\}\mathrm d x \; ,
\end{align*}
where $d_{\pm}(s)$ are defined in (\ref{eq:defd+-}).  For $s=0,\dots,h$, define
the Radon measure $\nu_s$ on $[0,1] \times {[-\infty,\infty]\setminus\{0\}}$~by
\begin{align*}
  \nu_{s}(\mathrm d t,\mathrm d x) & = \mathrm d t \, \lambda_{s}(\mathrm d x)  \; .
\end{align*}
Set ${\bf Y}_{n,i} = (a_n^{-1}Y_i,b_n^{-1} Y_iY_{i+1}, \ldots, b_n^{-1}
Y_iY_{i+h})$, where $a_n$ and $b_n$ are defined in~(\ref{eq:def-an}) and
(\ref{eq:def-bn}) respectively, and let $N_{n}$ be the point process defined on
$[0,1] \times ([-\infty,\infty]^{h+1} \setminus \{{\bf 0}\})$ by
\begin{align*}
  N_{n} = \sum_{i=1}^n \delta_{(i/n,{\bf Y}_{n,i})} \; ,
\end{align*}
where $\delta_x$ denotes the Dirac measure at $x$.

Our first result is that for the usual univariate point process of exceedances,
there is no effect of leverage. This is a consequence of the
asymptotic independence (\ref{eq:Breiman-0b}).
\begin{prop}
  \label{prop:pp-univarie}
  Let Assumption~\ref{hypo:iid-bivarie} hold and assume that $\sigma$ is a
  continuous function such that (\ref{eq:sigma-assumption}) holds with
  $q>\alpha$. Then $\sum_{i=1}^n \delta_{(i/n,{Y}_{i}/a_n)}$ converges weakly to
  a Poisson point process with mean measure $\nu_{0}$.
\end{prop}

For the multivariate point process $N_n$, we consider first LMSV models and then
models with leverage.

\subsection{Point process convergence: LMSV case}
\begin{prop}
  \label{prop:sv-pp}
  Let Assumption~\ref{hypo:iid-bivarie} hold and assume that the sequences
  $\{\eta_i\}$ and $\{Z_i\}$ are independent. Assume that the continuous
  volatility function $\sigma$ satisfies (\ref{eq:sigma-assumption}) for some
  $q>2\alpha$. Then
\begin{equation}
  \label{eq:pp-conv}
  N_{n} \Rightarrow  \sum_{i=0}^h \sum_{k=1}^{\infty} \delta_{(t_k,j_{k,i}\mathbf e_i)},
\end{equation}
where $\sum_{k=1}^{\infty} \delta_{(t_k,j_{k,0})}, \dots,
\sum_{k=1}^{\infty} \delta_{(t_k,j_{k,h})}$ are independent Poisson
processes with mean measures $\nu_{0},\ldots,\nu_{h}$, and
${\bf e}_i \in \mathbb R^{h+1}$ is the $i$-th basis component. Here,
$\Rightarrow$ denotes convergence in distribution in the space of
Radon point measures on $(0,1]\times[-\infty,\infty]^{h+1} \setminus
\{{\bf 0}\}$ equipped with the vague topology.
\end{prop}

\subsection{Point process convergence:  case of leverage}

\begin{prop}
  \label{prop:egarch-pp-expo}
  Let Assumption~\ref{hypo:iid-bivarie} hold. Assume that $\sigma(x) = \exp(x)$
  and $Z_0\exp(c \eta_0)$ is tail equivalent to $Z_0$ for all $c$.  Then the
  convergence~(\ref{eq:pp-conv}) holds.
\end{prop}

\begin{prop}
  \label{prop:egarch-pp-quadratic}
  Let Assumption~\ref{hypo:iid-bivarie} hold. Assume that the distribution of
  $(Z_0,\eta_0)$ and the function $\sigma$ satisfy the assumptions of
  Lemma~\ref{lemma:asymp-indep-EGARCH-quadratic} and moreover
  \begin{align}
    |\sigma(x+y) - \sigma(x+z)| \leq C (\sigma(x) \vee 1) \{(\sigma( y) \vee 1)
    +(\sigma(z) \vee 1) \} |y-z| \; . \label{eq:sigma-condition-truncation}
  \end{align}
  Assume that condition~(\ref{eq:sigma-assumption}) holds for some $q>2\alpha$.
  Then the convergence~(\ref{eq:pp-conv}) holds.
\end{prop}
The condition~(\ref{eq:sigma-condition-truncation}) is an ad-hoc condition which
is needed for a truncation argument used in the proof. It is satisfied by all
symmetric polynomials with positive coefficients. (The proof would not be
simplified by considering polynomials rather than functions satisfying this
assumption.)

\section{Partial Sums}
\label{sec:partial-sums}

Define
\begin{align*}
  S_n(t) & = \sum_{i=1}^{[nt]} Y_i \; , \ \ S_{p,n}(t) = \sum_{i=1}^{[nt]}
  |Y_i|^p \; .
\end{align*}
For any function $g$ such that $\mathbb{E}[g^2(\eta_0)]<\infty$ and any integer
$q\geq1$, define
\begin{align*}
  J_q(g) = \mathbb{E}[H_q(\eta_0)g(\eta_0)]\; ,
\end{align*}
where $H_q$ is the $q$-th Hermite polynomial. The Hermite rank
$\tau(g)$ of the function $g$ is the smallest positive integer
$\tau$ such that $J_\tau(g) \ne 0$.  Let $R_{\tau,H}$ be the
so-called Hermite process of order $\tau$ with self-similarity index
$1-\tau(1-H)$. See \cite{arcones:1994} or
Appendix~\ref{sec:LRD-Gaussian} for more details. Let $\stackrel{\scriptstyle \mathcal D}{\Rightarrow}$
denote convergence in the Skorokhod space $\mathcal D([0,1],\mathbb
R)$ of real valued right-continuous functions with left limits,
endowed with the $J_1$ topology, cf. \cite{whitt:2002}.
\begin{thm}
  \label{thm:partial-sums-egarch}
  Let Assumption~\ref{hypo:iid-bivarie} hold and assume that the function
  $\sigma$ is continuous and~(\ref{eq:sigma-assumption}) holds for some
  $q>2\alpha$.
  \begin{enumerate}[(i)]
  \item If $1 < \alpha < 2$ and $\mathbb{E}[Z_0]=0$, then $a_n^{-1} S_{n}$ converges
    weakly  in the space $\mathcal D([0,1),\mathbb R)$ endowed with
    Skorokhod's $J_1$ topology to an $\alpha$-stable L\'evy process with skewness~$2\beta-1$.
  \end{enumerate}
  Let $\tau_p=\tau(\sigma^p)$ be the Hermite rank of the function $\sigma^{p}$.
\begin{enumerate}[(i)] \addtocounter{enumi}{+1}
\item If $p<\alpha<2p$ and $1-\tau_p(1-H)<p/\alpha$, then
  \begin{equation}
    \label{eq:iid-convergence-egarch}
    a_n^{-p} (S_{p,n}-n\mathbb{E}[|Y_0|^p])  \stackrel{\scriptstyle \mathcal D}{\Rightarrow}  L_{\alpha/p} \; ,
  \end{equation}
  where $L_{\alpha/p}$ is a totally skewed to the right $\alpha/p$-stable L\'evy
  process.
\item If $p<\alpha<2p$ and $1-\tau_p(1-H)>p/\alpha$, then
  \begin{align}
    \label{eq:lrd-convergence-egarch}
    n^{-1} \rho_n^{-\tau_p/2} (S_{p,n} - n \mathbb{E}[|Y_0|^p]) \stackrel{\scriptstyle \mathcal D}{\Rightarrow}
    \frac{J_{\tau_p}(\sigma^p)\mathbb{E}[|Z_1|^p]} {\tau_p !} R_{\tau_p,H} \; .
  \end{align}
\item If $p> \alpha$, then $a_n^{-p} S_{p,n} \stackrel{\scriptstyle \mathcal D}{\Rightarrow} L_{\alpha/p}$, where
  $L_{\alpha/p}$ is a positive $\alpha/p$-stable L\'evy process.
\end{enumerate}

\end{thm}

Note that there is no effect of leverage. The situation will be
different for the sample covariances. The fact that when the marginal
distribution has infinite mean, long memory does not play any role and only a
stable limit can arise was observed in a different context by \cite{davis:1983}.

\section{Sample covariances}
\label{sec:sample-covariances}

In order to explain more clearly the nature of the results and the problems that
arise, we start by considering the sample covariances of the sequence $\{Y_i\}$,
without assuming that $\mathbb{E}[Z_0]=0$.  For notational simplicity, assume that we
observe a sample of length $n+h$. Assume that $\alpha>1$. Let $\bar
Y_n=n^{-1}\sum_{j=1}^nY_j$ denote the sample mean, $m=\mathbb{E}[Z_0]$,
$\mu_Y=\mathbb{E}[Y_0]=m\mathbb{E}[\sigma_0]$ and define the sample covariances by
\begin{align*}
  \hat\gamma_{n}(s) & = \frac1n \sum_{i=1}^{n} (Y_i-\bar Y_n) (Y_{i+s}-\bar Y_n)
  \; , \ 0 \leq s \leq h \; ,
\end{align*}
For simplicity, we have defined all the sample covariances as sums with the same
range of indices $1,\dots,n$. This obviously does not affect the asymptotic
theory.  For $s=0,\dots, h$, define furthermore
\begin{align*}
C_n(s) & = \frac1n \sum_{i=1}^{n} Y_iY_{i+s} \; .
\end{align*}
Then, defining $\gamma(s) = \mathrm{cov}(Y_0,Y_s)$, we have, for $s=0,\dots,h$,
\begin{align*}
  \hat \gamma_n(s) -\gamma(s) & = C_n(s) - \mathbb{E}[Y_0Y_s] + \mu_Y^2-\bar Y_n^2
  +O_P(1/n) \; .
\end{align*}
Under the assumptions of Theorem~\ref{thm:partial-sums-egarch}, $\bar
Y_n^2-\mu_Y^2 = O_P(a_n)$.  This term  never contributes to the limit.
Consider now $C_n(s)$. Recall that $\mathcal F_i$ is the sigma-field generated
by $(\eta_j,Z_j)$, $j\leq i$ and define
\begin{align*}
\hat{X}_{i,s} = \frac{\mathbb{E}[X_{i+s} \mid \mathcal
F_{i-1}]}{\mathrm{var}(\mathbb{E}[X_{i+s} \mid
    \mathcal F_{i-1}])} = \varsigma_s^{-1} \sum_{j=s+1}^\infty c_j \eta_{i+s-j} \; ,
\end{align*}
with $\varsigma_s^2 = \sum_{j=s+1}^\infty c_j^2$.  Let $K$ be the function
defined on $\mathbb R^2$ by
\begin{align}
  \label{eq:def-K}
  K(x,\hat x) = \mathbb{E}[Z_s] \mathbb{E} \left[ Z_0 \sigma(x) \sigma\left( \sum_{j=1}^s c_j
      \eta_{s-j} + \varsigma_s\hat x \right) \right] - \mathbb{E}[Y_0Y_s] \; .
\end{align}
Then, for each $i\ge 0$,  it holds that
\begin{align*}
  \mathbb{E}[ Y_iY_{i+s} \mid \mathcal F_{i-1} ] - \mathbb{E}[ Y_0Y_{s}] = K (X_i,\hat{X}_{i,s})
  \; .
\end{align*}
We see that if $m=\mathbb{E}[Z_s]=0$, then the function $K$ is identically
vanishing. We next write
\begin{align*}
  C_n(s) - \mathbb{E}[Y_0Y_s] & = \frac 1n \sum_{i=1}^{n} \{Y_iY_{i+s} -
  \mathbb{E}[Y_iY_{i+s} \mid \mathcal F_{i-1} ] \} + \frac1n \sum_{i=1}^{n}
  K(X_i,\hat{X}_{i,s})=\frac1n M_{n,s} + \frac 1n T_{n,s} \; .
\end{align*}
The point process convergence results of the previous section will allow to
prove that $b_n^{-1} M_{n,s}$ has a stable limit.  If $m=\mathbb{E}[Z]=0$, then this
will be the limit of $b_n^{-1}(C_n(s)- \mathbb{E}[Y_0Y_s])$, regardless of the
presence of leverage. We can thus state a first result. Let $\stackrel{\scriptstyle d}{\to}$ denote
weak convergence of sequences of finite dimensional random vectors.
\begin{thm}
  Assume that $\alpha\in (1,2)$ and $\mathbb{E}[Z_0]=0$.  Under the assumptions of
  Propositions~\ref{prop:sv-pp}, \ref{prop:egarch-pp-expo} or
  \ref{prop:egarch-pp-quadratic},
  \begin{align*}
    n b_n^{-1} (\hat\gamma_{n}(1)-\gamma(1),\dots,\hat\gamma_{n}(h) -\gamma(h))
    \stackrel{\scriptstyle d}{\to} (\mathcal L_1,\dots,\mathcal L_h) \; ,
  \end{align*}
  where $\mathcal L_1,\dots,\mathcal L_h$ are independent $\alpha$-stable random variables.
\end{thm}
This result was obtained by \cite{DavisMikosch2001} in the (LM)SV
case for the function $\sigma(x)=\exp(x)$ and under implicit
conditions that rule out long memory.

We continue the discussion under the assumption that $m\ne0$. Then
the term $T_{n,s}$ is the partial sum of a sequence which is a
function of a bivariate Gaussian sequence. It can be treated by
applying the results of \cite{arcones:1994}. Its rate of convergence
and limiting distribution will depend on the Hermite rank of the
function $K$ with respect to the bivariate Gaussian vector
$(X_0,\hat{X}_{0,s})$, which is fully characterized by the covariance
between $X_0$ and $\hat{X}_{0,s}$,
\begin{align*}
  \mathrm{cov}(X_0,\hat{X}_{0,s}) = \varsigma_s^{-1} \sum_{j=1}^\infty c_j c_{j+s} =
  \varsigma_s^{-1} \rho_s \; .
\end{align*}

\subsubsection*{LMSV case}
Since in this context the noise sequence $\{Z_i\}$ and the volatility sequence
$\{\sigma_i\}$ are independent, we compute easily that
\begin{align*}
  K(x,y) = m^2 \sigma(x) \mathbb{E}[\sigma(\varkappa_s \zeta + c_s \eta_0 +
  \varsigma_s y)] - m^2 \mathbb{E}[\sigma(X_0)\sigma(X_s)]\;,
\end{align*}
where $\varkappa_s^2 = \sum_{j=1}^{s-1} c_j^2$ and $\zeta$ is a standard
Gaussian random variable, independent of $\eta_0$. Thus, the Hermite rank of the
function $K$ depends only on the function $\sigma$ (but is not necessarily equal
to the Hermite rank of $\sigma$).

\subsubsection*{Case of leverage}
In that case, the dependence between $\eta_0$ and $Z_0$ comes into play. We now have
\begin{align*}
  K(x,y) = m \sigma(x) \mathbb{E}[\sigma(\varkappa_s \zeta + c_s \eta_0 + \varsigma_sy)Z_0] - m
  \mathbb{E}[\sigma(X_0)\sigma(X_s)Z_0] \; ,
\end{align*}
and now the Hermite rank of $K$ depends also on $Z_0$.  Different situations can
occur. We give two examples.
\begin{example}
  Consider the case $\sigma(x) = \exp(x)$. Then
  \begin{align*}
    \mathbb{E}[ Y_0Y_{s} \mid \mathcal F_{-1} ] & = \mathbb{E}[Z_0 Z_s \exp(X_0)
    \exp(X_s) \mid \mathcal F_{-1}]\\
& = m \mathbb{E}[Z_0 \exp(c_s \eta_0)]
\mathbb{E}\left[\exp\left(\sum_{j=1}^{s-1} c_j
      \eta_{s-j}\right) \right] \exp\left(X_0+\varsigma_s \hat{X}_{0,s}\right)  \; .
  \end{align*}
  Denote $\tilde m = \mathbb{E}[Z_0 \exp(c_s \eta_0)]$ and note that $
  \mathbb{E}\left[\exp\left(\sum_{j=1}^{s-1} c_j \eta_{s-j}\right) \right] = \exp\left(\varkappa_s^2/2\right)$. Thus
  \begin{align*}
    K(x,y) = m \tilde m \, \exp\left(\varkappa_s^2/2\right) \left\{\exp\left(x+\varsigma_sy\right) - \mathbb{E}\left[
    \exp\left(X_0+\varsigma_s \hat{X}_{0,s}\right) \right]\right\} \; .
  \end{align*}
  If $\mathbb{E}[Z_0]=0$ or $\mathbb{E}[Z_0\exp\left(c_s\eta_0\right)]=0$, then the function
  $K$ is identically vanishing and $T_{n,s}=0$.  Otherwise, the Hermite rank of
  $K$ with respect to $(X_0,\hat{X}_{0,s})$ is 1.  Thus, applying
  \cite[Theorem~6]{arcones:1994} (in the one-dimensional case) yields that
  $n^{-1} \rho_n^{-1/2} T_{n,s}$ converges weakly to a zero mean Gaussian
  distribution. The rate of convergence is the same as in the LMSV case but the
  asymptotic variance is different unless $\mathbb{E}[Z_0 \exp(c_s\eta_0)]
  =\mathbb{E}[Z_0] \mathbb{E}[ \exp(c_s\eta_0)]$.
\end{example}

\begin{example}
  Consider $\sigma(x)=x^2$.  Denote $\check{X}_{i,s} = \varkappa_s^{-1}
  \sum_{j=1}^{s-1} c_j \eta_{i+s-j}$.  Then
  \begin{align*}
    \mathbb{E}[ Y_0Y_{s} \mid \mathcal F_{-1} ] & = \mathbb{E}[Z_0 Z_s X_0^2 (\varkappa_s
    \check{X}_{0,s} + \varsigma_s \hat{X}_{0,s} + c_s \eta_0)^2  \mid \mathcal F_{-1}] \\
    & = m X_0^2 \left\{ \varkappa_s^2 m + c_s \mathbb{E}[Z_0\eta_0^2] + \varsigma_s m
      (\hat{X}_{0,s})^2 + 2\varsigma_s c_s \mathbb{E}[Z_0\eta_0] \hat{X}_{0,s} \right\}.
  \end{align*}
Thus
\begin{align*}
  K(x,y) & = \varsigma_s m^2 (x^2y^2 - \mathbb{E}[X_0^2(\hat{X}_{0,s}^2] ) + 2 \varsigma_s c_s
  m\mathbb{E}[Z_0\eta_0] \{x^2y-\mathbb{E}[X_0^2 \hat{X}_{0,s}]\} \\
  &  \hspace*{.5cm}+   (\varkappa_s^2m^2  + c_s m \mathbb{E}[Z_0\eta_0^2]) (x^2-1)
\end{align*}
and it can be verified that the Hermite rank of $K$ with respect to $(X_0,\hat
X_0^{(s)})$ is 1, except if $\mathbb{E}[Z_0\eta_0] = 0$, which holds in the LMSV
case. Thus we see that the rate of convergence of $T_{n,s}$ depends on the
presence or absence of leverage. See Example \ref{xmpl:egarch-different} for
details.
\end{example}
Let us now introduce the notations that will be used to deal with
sample covariances of powers. For $p>0$ define $m_p=\mathbb{E}[|Z_0|^p]$.
If
  $p\in(\alpha,2\alpha)$ and Assumption~(\ref{eq:model-2}) holds, $m_p$ is
  finite and $\mathbb{E}[|Z_0|^{2p}]=\infty$. Moreover, under the assumptions of
  Lemma~\ref{lem:asympt-indep-lmsv} or~\ref{lemma:asymp-indep-EGARCH-expo}, for
  $s>0$, $\mathbb{E}[|Y_0Y_s|^p]<\infty$ and $\mathbb{E}[|Y_0Y_s|^{2p}]=\infty$ for
  $p\in(\alpha/2,\alpha)$.  Thus the autocovariance $\gamma_p(s) =
  \mathrm{cov}(|Y_0|^p,|Y_s|^p)$ is well defined. Furthermore, define $\bar Y_{p,n} =
n^{-1} \sum_{i=1}^n |Y_i|^p$ and
\begin{align*}
  \hat \gamma_{p,n}(s) = \frac 1n \sum_{i=1}^{ n} (|Y_i|^p - \bar Y_{p,n})
  (|Y_{i+s}|^p - \bar Y_{p,n}) \; .
\end{align*}
Define the functions $K_{p,s}^*$ (LMSV case) and $K_{p,s}^\dag$ (case with
leverage) by
\begin{align}
  K_{p,s}^*(x,y) & = m_p^2 \sigma^p(x) \mathbb{E}[\sigma^p(\varkappa_s \zeta + c_s
  \eta_0 + \varsigma_s y)] -
  m_p^2  \mathbb{E}[\sigma^p(X_0)\sigma^p(X_s)] \; ,  \label{eq:def-K-lmsv} \\
  K_{p,s}^\dag(x,y) & = m_p \sigma^p(x) \mathbb{E}[\sigma^p(\varkappa_s \zeta + c_s
  \eta_0 + \varsigma_sy)|Z_0|^p] - m_p \mathbb{E}[\sigma^p(X_0)\sigma^p(X_s)|Z_0|^p]
  \; .  \label{eq:def-K-egarch}
\end{align}

\subsection{Convergence of the sample covariance of powers: LMSV case}

\begin{thm}
  \label{theo:cov-lmsv}
  Let Assumption~\ref{hypo:iid-bivarie} hold and assume that the sequences
  $\{\eta_i\}$ and $\{Z_i\}$ are independent. Let the function $\sigma$ be
  continuous and satisfy~(\ref{eq:sigma-assumption}) with $q>4\alpha$.  For a
  fixed  integer  $s\geq1$, let $\tau_p^*(s)$ be the Hermite rank of the
  bivariate function $K_{p,s}^*$ defined
  by~(\ref{eq:def-K-lmsv}), with respect to a bivariate Gaussian vector with
  standard marginal distributions and correlation $\varsigma_s^{-1}\gamma_s$.
  \begin{itemize}
  \item If $p<\alpha<2p$ and $1-\tau_p^*(s)(1-H)<p/\alpha$, then
    \begin{align*}
      n b_n^{-p} \hat\gamma_{p,n}(s)-\gamma_p(s) \stackrel{\scriptstyle d}{\to} \mathcal L_s \; ,
    \end{align*}
    where $\mathcal L_s$ is a $\alpha/p$-stable random variables.
  \item If $p<\alpha<2p$ and $1-\tau_p^*(s)(1-H)>p/\alpha$, then
    \begin{align*}
      \rho_n^{-\tau^*_p(s)/2}
      (\hat\gamma_{p,n}(s)-\gamma_p(s))
      \stackrel{\scriptstyle d}{\to} G_s^* \; ,
    \end{align*}
    where the random variable $G_s^*$ is Gaussian if $\tau_p^*(s)=1$.
  \end{itemize}
\end{thm}
For different values $s=1,\ldots,h$, the Hermite ranks $\tau_p^*(s)$ of the
functions $K_{p,s}^*$ may be different.  Therefore, in order to consider the joint
autovovariances at lags $s=1,\dots,h$, we define
$$
\tau_p^*=\min\{\tau_p^*(1),\ldots,\tau_p^*(h)\} \; .
$$
\addtocounter{cor}{+1}
\begin{cor}\label{cor:1}
  Under the assumptions of Theorem~\ref{theo:cov-lmsv},
\begin{itemize}
  \item If $1-\tau_p^*(1-H)<p/\alpha$, then
    \begin{align*}
      n b_n^{-p} (\hat\gamma_{p,n}(1)-\gamma_p(1),\dots,\hat\gamma_{p,n}(h)
      -\gamma_p(h)) \stackrel{\scriptstyle d}{\to} (\mathcal L_1,\dots,\mathcal L_h) \; ,
    \end{align*}
    where $\mathcal L_1,\dots,\mathcal L_p$ are independent $\alpha/p$-stable random variables.
  \item If $1-\tau_p^*(1-H)>p/\alpha$, then
    \begin{align*}
      \rho_n^{-\tau^*_p/2}
      (\hat\gamma_{p,n}(1)-\gamma_p(1),\dots,\hat\gamma_{p,n}(h) -\gamma_p(h))
      \stackrel{\scriptstyle d}{\to} (\tilde G_1^*,\dots,\tilde G_h^*) \; ,
    \end{align*}
    where $\tilde G_s^*=G_s^*$ if $\tau_p^*(s)=\tau_p^*$ and $\tilde
      G_s^*=0$ otherwise.
  \end{itemize}

\end{cor}
We see that the joint limiting vector $(\tilde G_1^*,\dots,\tilde
G_h^*)$
  may have certain zero components if there exist indices~$s$ such that
  $\tau_p^*(s)>\tau_p^*$. However, for standard choices of the function
  $\sigma$, the Hermite rank~$\tau_p^*(s)$ does not depend on~$s$. For instance,
  for $\sigma(x)=\exp(x)$, $\tau_p^*(s)=1$ for all $s$, and for $\sigma(x)=x^2$,
  $\tau_p^*(s)=2$ for all $s$.
\subsection{Convergence of sample covariance of powers:  case of leverage}

\begin{thm}
  \label{theo:cov-egarch}
  Let the assumptions of Proposition~\ref{prop:egarch-pp-expo} or
  \ref{prop:egarch-pp-quadratic} hold and assume
  that~(\ref{eq:sigma-assumption}) holds for some $q>4\alpha$. Let
  $\tau_p^\dag(s)$ be the Hermite rank of the bivariate function
  $K_{p,s}^\dag$ defined by~(\ref{eq:def-K-egarch}), with respect to a
  bivariate Gaussian vector with standard marginal distributions and correlation
  $\varsigma_s^{-1}\gamma_s$.
  \begin{itemize}
  \item If $p<\alpha<2p$ and $1-\tau_p^\dag(s)(1-H)<p/\alpha$, then
    \begin{align*}
      n b_n^{-p} (\hat\gamma_{p,n}(1)-\gamma_p(1),\dots,\hat\gamma_{p,n}(h)
      -\gamma_p(h)) \stackrel{\scriptstyle d}{\to} \mathcal L_s \; ,
    \end{align*}
    where $\mathcal L_s$ is a $\alpha/p$-stable random variable.
  \item If $p<\alpha<2p$ and $1-\tau_p^\dag(s)(1-H)>p/\alpha$, then
    \begin{align*}
      \rho_n^{-\tau^\dag_p(s)/2}
      (\hat\gamma_{p,n}(1)-\gamma_p(1),\dots,\hat\gamma_{p,n}(h) -\gamma_p(h))
      \stackrel{\scriptstyle d}{\to} G_s^\dag \; ,
    \end{align*}
    where the random vector $G_s^\dag$ is Gaussian if $\tau_p^\dag(s)=1$.
  \end{itemize}
\end{thm}
Again, as in the previous case, in order to formulate the multivariate result,
we define further
$$
\tau_p^\dag=\min\{\tau_p^\dag(1),\ldots,\tau_p^\dag(h)\} \; .
$$

\begin{cor}\label{cor:2}
  Under the assumptions of Theorem~\ref{theo:cov-egarch},
\begin{itemize}
  \item If $1-\tau_p^\dag(1-H)<p/\alpha$, then
    \begin{align*}
      n b_n^{-p} (\hat\gamma_{p,n}(1)-\gamma_p(1),\dots,\hat\gamma_{p,n}(h)
      -\gamma_p(h)) \stackrel{\scriptstyle d}{\to} (\mathcal L_1,\dots,\mathcal L_h) \; ,
    \end{align*}
    where $\mathcal L_1,\dots,\mathcal L_p$ are independent $\alpha/p$-stable random variables.
  \item If $1-\tau_p^\dag(1-H)>p/\alpha$, then
    \begin{align*}
      \rho_n^{-\tau^\dag_p/2}
      (\hat\gamma_{p,n}(1)-\gamma_p(1),\dots,\hat\gamma_{p,n}(h) -\gamma_p(h))
      \stackrel{\scriptstyle d}{\to} (\tilde G_1^\dag,\dots,\tilde G_h^\dag) \; ,
    \end{align*}
    where $\tilde G_s^\dag=G_s^\dag$ if $\tau_p^\dag(s)=\tau_p^\dag$ and $\tilde
      G_s^\dag=0$ otherwise.
  \end{itemize}
\end{cor}
The main difference between Theorems~\ref{theo:cov-lmsv} and
\ref{theo:cov-egarch} (or, Corollaries \ref{cor:1} and \ref{cor:2}) is the
Hermite rank considered. Under the conditions that ensure convergence to a
stable limit, the rates of convergence and the limits are the same in both
theorems. Otherwise, the rates and the limits may be different.

\begin{example}
  \label{xmpl:exponential}
  Consider the case $\sigma(x)=\exp(x)$. For all $s\ge 1$ we have
  $\tau_p^\dag=\tau_p^\dag(s)=1$. Thus, under the assumptions of
Theorem~\ref{theo:cov-egarch}, we have:
\begin{itemize}
\item If $H<p/\alpha$, then $n b_n^{-1} \{\hat\gamma_{p,n}(s)-\gamma_p(s)\}$
  converges weakly to a stable law.
\item If $H>p/\alpha$, then $\rho_n^{-1/2}
  \{\hat\gamma_{p,n}(s)-\gamma_p(s)\}$ converges weakly to a zero mean Gaussian
  distribution.
\end{itemize}
The dichotomy is the same as in the LMSV case, but the variance of
the limiting distribution in the case $H>p/\alpha$ is different
except if $\mathbb{E}[Z_0 \exp(c_s\eta_0)] =\mathbb{E}[Z_0] \mathbb{E}[
\exp(c_s\eta_0)]$.
\end{example}

\begin{example}
  \label{xmpl:egarch-different}
  Consider the case $\sigma(x)=x^2$ and $p=1$. Assume that
  $\mathbb{E}[\eta_1|Z_1|]\not=0$. Then for each $s\ge 1$,
    $\tau_1^\dag=\tau_1^\dag(s)=1$ whereas $\tau_p^*=\tau_p^*(s)=2$, thus the
  dichotomy is not the same as in the LMSV case and the rate of convergence
  differs in the case $H>1/\alpha$.
\begin{itemize}
\item If $H < 1/\alpha$, then $n b_n^{-1} \{ \hat\gamma_{n,1}(s) -
  \gamma_1(s)\}$ converges weakly to a stable law.
\item If $H > 1/\alpha$, then $\rho_n^{-1/2} \{ \hat\gamma_{n,1}(s) -
  \gamma_1(s)\}$ converges weakly to a zero mean Gaussian distribution.
\end{itemize}
If we assume now that $\mathbb{E}[\eta_1|Z_1|]=0$, then
$\tau_1^\dag=\tau_p^*=2$. Thus the dichotomy is the same as in the LMSV case,
but the limiting distribution in the non stable case can be different from the
one in the LMSV case.
\begin{itemize}
\item If $2H-1 < 1/\alpha$, then $n b_n^{-1} \{\hat\gamma_{1,n}(s) -
  \gamma_1(s)\}$ converges weakly to a stable law.
\item If $2H-1 > 1/\alpha$, then $\rho_n^{-1} \{\hat\gamma_{1,n}(s) -
  \gamma_1(s)\}$ converges weakly to a zero mean non Gaussian distribution.
\end{itemize}
If moreover $\mathbb{E}[H_2(\eta_1)| Z_1|]=0$, then for each $s$,
the functions $K_{p,s}^*$ and $K_{p,s}^\dag$ are equal, and thus
the limiting distribution is the same as in the LMSV case.
\end{example}

\section{Proofs}\label{sec:proofs}

\begin{lem}
  \label{lem:tail-equivalence}
  Let $Z$ be a nonnegative random variable with a regularly varying right tail
  with index $-\alpha$, $\alpha>0$. Let $g$ be a bounded function on
  $[0,\infty)$ such that $\lim_{x\to+\infty} g(x) = c_g \in(0,\infty)$. Then
  $Zg(Z)$ is tail equivalent to $Z$:
  \begin{align*}
    \lim_{x\to+\infty} \frac{\mathbb{P}(Zg(Z)>x)}{\mathbb{P}(Z>x)} = c_g^\alpha \; .
  \end{align*}
\end{lem}

\begin{proof}
  Fix some $\epsilon>0$ and let $x_0$ be large enough so that
  $|g(x)-c_g|/c_g<\epsilon$ for all $x>x_0$.  The function $g$ is bounded, thus $zg(z)>x$
  implies that $z>x/\|g\|_\infty$ and if $x>x_0\|g\|_\infty$, we have
  \begin{align*}
    \mathbb{P}(Zg(Z)>x) & = \mathbb{P}(Zg(Z)>x,Z>x/\|g\|_\infty)   \\
    & \leq \mathbb{P}(Zc_g(1+\epsilon)>x,Z>x/\|g\|_\infty) \leq \mathbb{P}(Zc_g(1+\epsilon)>x) \;    .
  \end{align*}
  This yields the  upper bound:
\begin{align*}
  \limsup_{x\to+\infty} \frac{ \mathbb{P}(Zg(Z)>x)}{\mathbb{P}(Z>x) } \leq \limsup_{x\to+\infty}
  \frac{ \mathbb{P}(Zc(1+\epsilon)>x)}{\mathbb{P}(Z>x) } = c_g^\alpha(1+\epsilon)^\alpha \; .
\end{align*}
Conversely, we have
\begin{align*}
  \mathbb{P}(Zg(Z)>x) & = \mathbb{P}(Zg(Z)>x,Z>x/\|g\|_\infty)
  \geq \mathbb{P}(Zc_g(1-\epsilon)>x,Z>x/\|g\|_\infty) \\
  & = \mathbb{P}\left( Z > x \max\left\{ \frac1{c_g(1-\epsilon)} , \frac1{\|g\|_\infty}
    \right\} \right) = \mathbb{P}\left( Z > \frac x{c_g(1-\epsilon)}  \right)
\end{align*}
where the last equality comes from the fact that $(1-\epsilon)c_g
\leq c_g = \lim_{z\to+\infty} g(z) \leq \|g\|_\infty$.  Thus
\begin{align*}
  \liminf_{x\to+\infty} \frac{ \mathbb{P}(Zg(Z)>x)}{\mathbb{P}(Z>x) } \geq \limsup_{x\to+\infty}
  \frac{ \mathbb{P}(Zc_g(1-\epsilon)>x)}{\mathbb{P}(Z>x) } = c_g^\alpha(1-\epsilon)^\alpha \; .
\end{align*}
  Since $\epsilon$ is arbitrary, we obtain the desired limit.
\end{proof}

\begin{lem}
  \label{lem:bound-potter}
  Let $Z$ be a nonnegative random variable with a regularly varying right tail
  with index $-\alpha$, $\alpha>0$.  For each $\epsilon>0$, there exists a
  constant $C$, such that for all $x\geq1$ and all $y>0$,
  \begin{align}
    \frac{ \mathbb{P}(yZ>x) }{ \mathbb{P}(Z>x) } \leq C (y\vee1)^{\alpha+\epsilon} \; . \label{eq:claim-potter}
  \end{align}
\end{lem}

\begin{proof}
  If $y \leq 1$, then $\mathbb{P}(yZ>x) \leq \mathbb{P}(Z>x)$ so the requested bound holds
  trivially with $C=1$. Assume now that $y \geq 1$. Then, by Markov's inequality,
  \begin{align}
    \mathbb{P}(yZ > x) & = \mathbb{P}(Z>x) + \mathbb{P}(Z \mathbf 1_{\{Z \leq x\}} >x/y)
    \leq x^{-\alpha-\epsilon} y^{\alpha+\epsilon} \mathbb{E}[ Z^{\alpha+\epsilon}
    \mathbf 1_{\{Z \leq x\}} ] \; . \label{eq:decomp-yz>x}
  \end{align}
  Next, by \cite[Theorem~VIII.9.2]{feller:1971} or
  \cite[Theorem~8.1.2]{bingham:goldie:teugels:1989},
\begin{align*}
  \lim_{x\to+\infty} \frac{\mathbb{E}[Z^{\alpha+\epsilon} \mathbf 1_{\{Z \leq x\}}]}
  {x^{\alpha+\epsilon} \mathbb{P}(Z>x)} = \frac\alpha\epsilon  \; .
\end{align*}
Moreover, the function $x\to\mathbb{P}(Z>x)$ is decreasing on $[0,\infty)$, hence
bounded away from zero on compact sets of $[0,\infty)$.  Thus, there exists a
constant $C$ such that for all $x\geq1$,
\begin{align}
  \frac{\mathbb{E}[Z^{\alpha+\epsilon} \mathbf 1_{\{Z \leq x\}}]} {\mathbb{P}(Z>x)} \leq C
  x^{\alpha+\epsilon} \; . \label{eq:borne-feller}
\end{align}
Plugging~(\ref{eq:borne-feller}) into~(\ref{eq:decomp-yz>x}) yields, for all
$x,y\geq1$,
\begin{align*}
  \frac{ \mathbb{P}(yZ > x) }{ \mathbb{P}(Z>x) }& = 1 + C y^{\alpha+\epsilon}  \; .
\end{align*}
This concludes the proof of~(\ref{eq:claim-potter}).
\end{proof}

\begin{proof}[Proof of Lemma~\ref{lem:asympt-indep-lmsv}]
  Under the assumption of independence between the sequences $\{Z_i\}$ and
  $\{\eta_i\}$, as already mentioned, $Y_0$ is tail equivalent to $Z_0$ and
  $Y_0Y_h$ is tail equivalent to $Z_0Z_1$ for all $h$. The
  properties~(\ref{eq:defd+-}), (\ref{eq:def-bn}), (\ref{eq:domination}) are
  straightforward. We need to prove~(\ref{eq:indep1})
  and~(\ref{eq:indep-products}).   Since $Z_0$ is
  independent of $\sigma_j$ and $Z_j$, by conditioning, we have
  \begin{align*}
    n \mathbb{P}(|Y_0| > a_n x , |Y_0Y_j| > b_n x) & = \mathbb{E} \left[ n \bar F_{|Z|}
      \left( \frac{a_nx}{\sigma_0} \vee \frac {b_n x} {\sigma_0\sigma_j|Z_j|} \right)
    \right]
  \end{align*}
  with $F_{|Z|}$ the distribution function of $|Z_0|$.  Since $a_n/b_n\to0$, for
  any $y>0$, it holds that $\lim_{n\to+\infty} n\bar F_{|Z|}(b_ny)=0$. Thus,
\begin{align*}
  n \bar F_{|Z|} \left( \frac{a_nx}{\sigma_0} \vee \frac {b_n x}
    {\sigma_0\sigma_j|Z_j|} \right) \leq n \bar F_{|Z|} \left( \frac {b_n x}
    {\sigma_0\sigma_j|Z_j|} \right) \to 0 \; , \mbox{ a.s.}
\end{align*}
Moreover, by Lemma~\ref{lem:bound-potter} and the definition of $a_n$, for any
$\epsilon>0$ there exists a constant $C$ such that
\begin{align*}
  n \bar F_{|Z|} \left( \frac{a_nx}{\sigma_0} \vee \frac {b_n x}
    {\sigma_0\sigma_j|Z_j|} \right) \leq n \bar F_{|Z|} \left(
    \frac{a_nx}{\sigma_0} \right) \leq C x^{-\alpha-\epsilon} \sigma_0^{\alpha+\epsilon} \; .
\end{align*}
By assumption,~(\ref{eq:sigma-assumption}) holds for some
$q>\alpha$. Thus, choosing $\epsilon$ small enough allows to apply
the bounded convergence theorem and this proves~(\ref{eq:indep1}).
Next, to prove~(\ref{eq:indep-products}), note that $|Y_i| \wedge
|Y_j| \leq (\sigma_i\vee\sigma_j) (|Z_i|\wedge |Z_j|)$. Thus,
applying Lemma~\ref{lem:bound-potter}, we have
  \begin{align*}
    \mathbb{P}(|Y_0Y_i|>x,|Y_0Y_j|>x ) & = \mathbb{P}(|Z_0| \sigma_0(\sigma_i |Z_i| \wedge
    \sigma_j |Z_j|) > x) \\
    & \leq C \mathbb{P}(|Z_0|>x) \mathbb{E}[\sigma_0^{\alpha+\epsilon}
    (\sigma_i\vee\sigma_j)^{\alpha+\epsilon}] \mathbb{E}[(|Z_i|\wedge
    |Z_j|)^{\alpha+\epsilon}] \; .
  \end{align*}
  The expectation $\mathbb{E}[\sigma_0^{\alpha+\epsilon}
  (\sigma_i\vee\sigma_j)^{\alpha+\epsilon}] $ is finite for $\epsilon$ small
  enough, since Assumption~(\ref{eq:sigma-assumption}) holds with
  $q>2\alpha$. Since $\mathbb{P}(|Z_0|>x) = o(\mathbb{P}(|Z_1Z_2|>x))$, this
  yields~(\ref{eq:indep-products}) in the LMSV case.
\end{proof}

\begin{proof}[Proof of Lemma~\ref{lemma:asymp-indep-EGARCH-expo}]
  It suffices to prove the lemma when the random variables $Z_i$ are
  nonnegative.  Under the assumption of the Lemma, $\exp(c_h \eta_0) Z_0$ is
  tail equivalent to $Z_0$. Thus, by the Corollary in
  \cite[p.~245]{embrechts:goldie:1980}, $Z_0\exp(c_h\eta_0) Z_h$ is regularly
  varying with index $\alpha$ and tail equivalent to $Z_0Z_h$.  Since
  $\mathbb{E}[Z_0^\alpha]=\infty$, it also holds that $\mathbb{P}(Z_0>x) = o(\mathbb{P}(\exp(c_h
  \eta_0) Z_0Z_1>x))$, cf. \cite[Equation~(3.5)]{DavisResnick1986}.

  Define $\hat X_h = \sum_{k=1, k\not=h}^{\infty} c_k \eta_{h-k}$. Then $\hat
  X_h$ is independent of $Z_0$, $\eta_0$ and $Z_h$.  Since $Y_0Y_h =
  \exp(X_0+\hat X_h) Z_0\exp(c_h\eta_0) Z_h$, we can apply Breiman's Lemma
  to obtain that $Y_0Y_h$ is tail equivalent to $Z_0\exp(c_h\eta_0) Z_h$,
  hence to $Z_0Z_1$. Thus~(\ref{eq:domination}) and (\ref{eq:defd+-}) hold with
  \begin{align*}
    d_+(h) = \tilde \beta \frac{\mathbb{E}[\exp(\alpha(X_0+\hat
        X_h))]}{\mathbb{E}[\exp(\alpha(X_0+\hat X_1))]} \; , \ \ d_-(h) = (1-\tilde
    \beta) \frac{\mathbb{E}[\exp(\alpha(X_0+\hat X_h))]}{\mathbb{E}[\exp(\alpha(X_0+\hat X_1))]} \; ,
  \end{align*}
  where $\tilde \beta$ is the skewness parameter of
  $Z_0\exp(c_h\eta_0)
  Z_h$.

  We now prove~(\ref{eq:indep-products}).  For fixed $i,j$ such that $0<i<j$,
  define
  \begin{align*}
    \hat\sigma_i = \sigma(\hat X_i) = \exp\left(\sum_{k=1, k\not=i}^{\infty} c_k
      \eta_{i-k}\right) \; , \ \ \check\sigma_{i,j} = \sigma(\check{X}_{i,j}) =
    \exp\left(\sum_{k=1, k\not=j,j-i}^{\infty} c_k \eta_{j-k}\right) \; .
\end{align*}
Denote $\tilde Z_0^{(k)} = Z_0\exp(c_k \eta_0)$ and $V_i = \exp(c_{j-i}\eta_i)$.
Then
  \begin{align*}
    \mathbb{P}(Y_0Y_i > x , Y_0Y_j > x) & = \mathbb{P}(\sigma_0 \hat\sigma_i \tilde Z_0^{(i)} Z_i
    >x \; , \  \sigma_0  \check\sigma_{i,j} \tilde Z_0^{(i)}  \exp(c_{j-i}\eta_i) Z_j >  x ) \\
    & \leq \mathbb{P}(\sigma_0 (\hat \sigma_i \vee \check\sigma_{i,j}) (\tilde
    Z_0^{(i)}+\tilde Z_0^{(j)}) (Z_i \wedge V_i Z_j) > x) \; .
  \end{align*}
  Now, $(Z_i\wedge V_i Z_j)$ is independent of $\sigma_0 (\hat \sigma_i \vee
  \check\sigma_{i,j}) (\tilde Z_0^{(i)}+\tilde Z_0^{(j)})$, which is tail
  equivalent to $Z_0$ by assumption and Breiman's Lemma. Thus, in order to
  prove~(\ref{eq:indep-products}), we only need to show that for some
  $\delta>\alpha$, $\mathbb{E}[(Z_i \wedge V_i Z_j)^\delta]<\infty$. This is
  true. Indeed, since $\mathbb{E}[V_i^q]<\infty$ for all $q>1$, we can apply
  H\"older's inequality with $q$ arbitrarily close to 1. This yields for
  $p^{-1}+q^{-1}=1$,
  \begin{align*}
    \mathbb{E}[(Z_i \wedge V_i Z_j)^\delta] \leq \mathbb{E}[(1\vee V_i)^\delta (Z_i \wedge
    Z_j)^\delta] \leq \mathbb{E}^{1/p}[(1\vee V_i)^{p\delta}] \, \mathbb{E}^{1/q}[(Z_i
    \wedge Z_j)^{q\delta}] \; .
  \end{align*}
  The tail index of $(Z_i \wedge Z_j)$ is $2\alpha$, and thus $\mathbb{E}^{1/q}[(Z_i
  \wedge Z_j)^{q\delta}] < \infty$ for any $q$ and $\delta$ such that
  $q\delta<2\alpha$. Thus $\mathbb{E}[(Z_i \wedge V_i Z_j)^\delta]<\infty$ for any
  $\delta \in (\alpha,2\alpha)$ and~(\ref{eq:indep-products}) holds. The proof
  of~(\ref{eq:indep1}) is similar.
\end{proof}

\begin{proof}[Proof of Lemma~\ref{lemma:asymp-indep-EGARCH-quadratic}]
  We omit the proof of the regular variation and the tail equivalence between
  $Y_0Y_h$ and $Z_0Z_1$ which is a straightforward consequence of the
  assumption. We prove~(\ref{eq:indep-products}).  Using the notation of the
  proof of Lemma~\ref{lemma:asymp-indep-EGARCH-expo}, by the subadditivity
  property of $\sigma$, we have, for $j>i>0$, and for some constant $C$,
  \begin{align*}
    \mathbb{P}&(Y_0Y_i > x , Y_0Y_j > x) \\
    & = \mathbb{P}(\sigma_0 \sigma(\hat X_i + c_i \eta_0)
    Z_0 Z_i >x \; , \    \sigma_0  \sigma(\check{X}_{i,j} + c_j \eta_0 + c_{j-i} \eta_i)   Z_0   Z_j >  x \} \\
    & \leq \mathbb{P}(C \sigma_0 |Z_0| \{\sigma(\hat X_i) + \sigma(c_i\eta_0)\}
    \{\sigma(\check{X}_{i,j}) + \sigma(c_j\eta_0) + \sigma(c_{j-i} \eta_i)\} (|Z_i| \wedge |Z_j|) > x)    \\
    & \leq \mathbb{P}(C \sigma_0 |Z_0| \sigma(\hat X_i) \sigma(\check{X}_{i,j})(|Z_i| \wedge
    |Z_j|) > x) + \mathbb{P}(C \sigma_0  |Z_0| \sigma(\hat X_i) \sigma(c_j\eta_0) (|Z_i| \wedge |Z_j|) > x) \\
    & + \mathbb{P}(C \sigma_0 |Z_0| \sigma(\hat X_i) \sigma(c_{j-i} \eta_i) (|Z_i| \wedge
    |Z_j|)  > x) + \mathbb{P}(C \sigma_0 |Z_0|  \sigma(c_i\eta_0)\sigma(\check{X}_{i,j})  (|Z_i| \wedge |Z_j|) > x)    \\
    & + \mathbb{P}(C \sigma_0 |Z_0| \sigma(c_i\eta_0) \sigma(c_j\eta_0) (|Z_i| \wedge |Z_j|)
    > x) + \mathbb{P}(C \sigma_0 |Z_0| \sigma(c_i\eta_0) \sigma(c_{j-i} \eta_i) (|Z_i|
    \wedge |Z_j|) > x) \; .
  \end{align*}
  Now, under the assumptions of the Lemma, each of the last six probabilities
  can be expressed as $\mathbb{P}(\tilde Z U>x)$, where $\tilde Z$ is tail equivalent
  to $Z_0$ and $U$ is independent of $\tilde Z$ and $\mathbb{E}[|U|^q]<\infty$ for some
  $q>\alpha$. Thus, by Breiman's Lemma, $\tilde ZU$ is also tail equivalent to
  $Z_0$, and thus $ \mathbb{P}(Y_0Y_i > x , Y_0Y_j > x) = O(\mathbb{P}(|Z_0|>x)) =
  o(\mathbb{P}(|Y_0Y_1|>x))$, which proves~(\ref{eq:indep-products}).
\end{proof}

\subsection{Proof of Propositions~\ref{prop:pp-univarie}, \ref{prop:sv-pp},
  \ref{prop:egarch-pp-expo} and~\ref{prop:egarch-pp-quadratic}}

We omit some details of the proof, since it is a slight modification of the proof of Theorems
3.1 and 3.2 in \cite{DavisMikosch2001}, adapted to a general stochastic
volatility with possible leverage and long memory. Note that the proof of
\cite[Theorem 3.2]{DavisMikosch2001} refers to the proof of Theorem 2.4 in
\cite{davis:resnick:1985l}. The latter proof uses condition (2.6) in
\cite{davis:resnick:1985l}, which rules out long memory.

The proof is in two steps. In the first step we consider an $m$-dependent
approximation $X^{(m)}$ of the Gaussian process and prove point-process
convergence for the corresponding stochastic volatility process $Y^{(m)}$ for
each fixed $m$.  The second step naturally consists in proving that the limits
for the $m$-dependent approximations converge when $m$ tends to infinity, and
that this limit is indeed the limit of the original sequence.

\subsubsection*{First step}
Let $X_i^{(m)}=\sum_{k=1}^{m}c_k\eta_{i-k}$,
$Y_i^{(m)}=\sigma(X_i^{(m)})Z_i$ and define accordingly ${\bf
Y}_{n,i}^{(m)}$. Note that the tail properties
  of the process $\{Y^{(m)}_i\}$ are the same as those of the process $\{Y_i\}$,
  since the latter are proved without any particular assumptions on the
  coefficients $c_j$ of the expansion~(\ref{eq:linear}) apart from square
  summability.  In order to prove the desired point process convergence, as in
the proof of \cite[Theorem~3.1]{DavisMikosch2001}, we must check the following
two conditions (which are Equations (3.3) and (3.4) in \cite{DavisMikosch2001}):
  \begin{align}
    &    \mathbb{P}(\mathbf Y_{n,1}^{(m)} \in \cdot ) \stackrel{\scriptstyle v}{\to} \boldsymbol\nu_m \; ,
   \label{eq:conv-vague} \\
    & \lim_{k\to+\infty} \limsup_{n\to+\infty} n \sum_{i=2}^{[n/k]} \mathbb{E}[g({\bf
      Y}_{n,1}^{(m)}) g({\bf Y}_{n,i}^{(m)}) ] = 0 \;, \label{eq:Dprime}
  \end{align}
  where $\boldsymbol\nu_m$ is the mean measure of the limiting point process and
  (\ref{eq:Dprime}) must hold for any continuous bounded function $g$, compactly
  supported on $[0,1] \times [-\infty,\infty]^{h}\setminus\{\boldsymbol0\}$.

  The convergence~(\ref{eq:conv-vague}) is a straightforward consequence of the
  joint regular variation and the asymptotic independence
  properties~(\ref{eq:indep1}), (\ref{eq:indep-products}) of
  $Y_0,Y_0Y_1,\dots,Y_0Y_h$.  Let us now prove~(\ref{eq:Dprime}). Note first
  that, because of asymptotic independence, for any fixed $i$,
\begin{align*}
  \lim_{n\to+\infty} n \mathbb{E}[g({\bf Y}_{n,1}^{(m)}) g({\bf Y}_{n,i}^{(m)}) ] = 0  \; .
\end{align*}
Next, by $m$-dependence, for each $k$, as $n\to+\infty$, we have
\begin{align*}
  n \sum_{i=2+m+h}^{[n/k]} \mathbb{E}[g({\bf Y}_{n,1}^{(m)}) g({\bf Y}_{n,i}^{(m)}) ]
  & = n \sum_{i=2+m+h}^{[n/k]} \mathbb{E}[g({\bf Y}_{n,1}^{(m)})] \mathbb{E}[ g({\bf
    Y}_{n,i}^{(m)})] \\ & \sim \frac 1k \left( n\mathbb{E}[g({\bf Y}_{n,1}^{(m)})]
  \right)^2 \to \frac 1k \left(\int g \mathrm d \boldsymbol\nu_m \right) ^2 \; .
\end{align*}
This yields~(\ref{eq:Dprime}). Thus, we obtain that
$$
\sum_{i=1}^n \delta_{(i/n,{\bf Y}_{n,i}^{(m)})} \Rightarrow \sum_{l=1}^h
\sum_{k=1}^{\infty} \delta_{(t_k,j_{k,l}^{(m)}{\bf e}_l)},
$$
where $\sum_{k=1}^{\infty} \delta_{(t_k,j_{k,0}^{(m)})},
\dots,\sum_{k=1}^{\infty} \delta_{(t_k,j_{k,h}^{(m)})}$ are
independent Poisson processes with respective mean measures
\begin{align}
  \lambda_{0,m}(\mathrm d x) = \alpha \left\{ \beta_{m} x^{-\alpha-1} \mathbf
    1_{(0,\infty)}(x) + (1-\beta_{ m}) (-x)^{-\alpha-1} \mathbf
    1_{(-\infty,0)}(x) \right\} \mathrm d x \; , \label{mean-measure-1}
  \\
  \lambda_{s,m}(\mathrm d x) = \alpha \left\{ d_+^{(m)}(s) x^{-\alpha-1} \mathbf
    1_{(0,\infty)}(x) + d_-^{(m)}(s) (-x)^{-\alpha-1} \mathbf 1_{(-\infty,0)}(x)
  \right\}\mathrm d x \; ,\label{mean-measure-2}
\end{align}
where $d_+^{(m)}(s)$ and $d_-^{(m)}(s)$ depend on the process considered and
$\beta_m$ = ${\beta \mathbb{E}[\sigma^\alpha(X^{(m)})]/\mathbb{E}[\sigma^\alpha(X)]}$.

\subsubsection*{Second step}
We must now prove that
\begin{align}
  \label{eq:3.13}
  N_m \Rightarrow N
\end{align}
as $m\to+\infty$ and that for all $\eta>0$,
\begin{align}
  \label{eq:3.14}
  \lim_{m\to+\infty}
\limsup_{n\to+\infty} \mathbb{P}(\varrho(N_n,N_n^{(m)})>\eta) = 0 \; .
\end{align}
where $\varrho$ is the metric inducing the vague topology.  Cf. (3.13) and
(3.14) in \cite{DavisMikosch2001}.  To prove~(\ref{eq:3.13}), it suffices to
prove that
\begin{align}
  \lim_{m\to+\infty} \beta_m & = \beta \; , \label{eq:conv-betam} \\
\ \lim_{m\to+\infty} d_+^{(m)}(s) & = d_+(s) \; ,
  \ \lim_{m\to+\infty} d_-^{(m)}(s)=d_-(s) \; . \label{eq:conv-d+-m}
\end{align}
To prove~(\ref{eq:3.14}), as in the proof of
\cite[Theorem~3.3]{DavisMikosch2001}, it suffices to show that for
all $\epsilon>0$,
\begin{align}
  & \lim_{m\to+\infty} \limsup_{n\to+\infty} n \mathbb{P} (a_n^{-1} |Y_0-Y_{0}^{(m)}| >
  \epsilon) = 0 \; , \label{eq:pp-tightness1}  \\
  & \lim_{m\to+\infty} \limsup_{n\to+\infty} n \mathbb{P} \left(b_n^{-1} |Y_0Y_{s} -
    Y_0^{(m)}Y_s^{(m)}| > \epsilon \right) = 0 \; . \label{eq:pp-tightness2}
\end{align}
If~(\ref{eq:sigma-assumption}) holds for some $q>\alpha$ and if $\sigma$ is
continuous, then~(\ref{eq:conv-betam}) holds by bounded convergence, in both the
LMSV case and the case of leverage.
We now prove~(\ref{eq:pp-tightness1}). Since $Y_0$ and $Z_0$ are tail
equivalent, by Breiman's Lemma, we have
\begin{align*}
  \limsup_{n\to+\infty} n \mathbb{P} (a_n^{-1} |Y_0-Y_{0}^{(m)}| > \epsilon) \leq C
  \epsilon^{-\alpha} \mathbb{E}[|\sigma(X_0^{(m)})-\sigma(X_0)|^{\alpha}] \; .
\end{align*}
Continuity of $\sigma$, Assumption~(\ref{eq:sigma-assumption}) with $q>\alpha$
and the bounded convergence theorem imply that $\lim_{m\to+\infty}
\mathbb{E}[|\sigma(X_0^{(m)})-\sigma(X_0)|^{\alpha}]=0$. This
proves~(\ref{eq:pp-tightness1}) in both the LMSV case and the case of leverage.
We now split the proof of~(\ref{eq:conv-d+-m}) and~(\ref{eq:pp-tightness2})
between the LMSV and leverage cases.

\subsubsection*{LMSV case.}
In this case, we have
\begin{align*}
  d_+^{(m)}(s) = d_+(s)
  \frac{\mathbb{E}[\sigma^\alpha(X_0^{(m)})\sigma^\alpha(X_s^{(m)})]}
  {\mathbb{E}[\sigma^\alpha(X_0)\sigma^\alpha(X_s)]} \; , \ d_-^{(m)}(s) = d_-(s)
  \frac{\mathbb{E}[\sigma^\alpha(X_0^{(m)})\sigma^\alpha(X_s^{(m)})]}
  {\mathbb{E}[\sigma^\alpha(X_0)\sigma^\alpha(X_s)]} \; .
\end{align*}
For $s = 1,\dots,h$, define
$$
W_{m,s} = \sigma(X_0^{(m)})\sigma(X_{s}^{(m)}) - \sigma(X_1)\sigma(X_{1+s}) \; .
$$
Continuity of $\sigma$ implies that  $W_{m,s} \stackrel{\scriptstyle P}{\to} 0$
as $m\to+\infty$.  Under the Gaussian assumption, $X^{(m)} \stackrel d = u_m X$
for some $u_m\in(0,1)$, thus if~(\ref{eq:sigma-assumption}) holds for some
$q'>\alpha$, then it also holds that
\begin{align*}
  \sup_{m\geq1} \mathbb{E}[\sigma^{q'}(X^{(m)})] <\infty \; ,
\end{align*}
hence $W_m$ converges to 0 in $L^q$ for any $q<q'$.  Likewise, since
assumption (\ref{eq:sigma-assumption}) holds for some $q'>2\alpha$,
$W_{m,s}$ converges to 0 in $L^q$ for any $q<q'$.  Since $|W_m|$ and
$|W_{m,s}|$ converge to 0 in $L^\alpha$, we obtain that
$d_+^{(m)}(s)$ and $d_-^{(m)}(s)$ converge to the required limits.
We now prove~(\ref{eq:pp-tightness2}). Since $Z_0Z_s$ is tail
equivalent to $Y_0Y_1$, by another application of Breiman's Lemma,
we obtain, for $s=1,\ldots,h$ and $\epsilon>0$,
\begin{align*}
  \limsup_{n\to+\infty} \mathbb{P} (b_n^{-1} |Y_0Y_{s}-Y_{0}^{(m)}Y_{s}^{(m)}| >
  \epsilon ) & \le \limsup_{n\to+\infty} n \mathbb{P} \left(b_n^{-1} |Z_0Z_{s}| |W_{m,s}|
    > C \epsilon \right) \leq C^{-\alpha} \epsilon^{-\alpha}
  \mathbb{E}[|W_{m,s}|^{\alpha}]
\end{align*}
which converges to 0 as $m\to+\infty$. This concludes the proof of~(\ref{eq:pp-tightness2}) in the LMSV case.

To prove (\ref{eq:pp-tightness2}) in the case of leverage, we
further split the proof between the cases $\sigma(x)=\exp(x)$ and
$\sigma$ subadditive.

\subsubsection*{Case of leverage, $\sigma(x)=\exp(x)$}

Define $\hat X_s = \sum_{j=1 , {j\not= s}}^\infty c_j \eta_{s-j}$, $\hat
X_s^{(m)} = \sum_{j=1 , {j\not= s}}^m c_j \eta_{s-j}$ and
\begin{align*}
  \tilde W_{m,s} = |\exp(X_0+\hat X_s) - \exp(X_0^{(m)}+\hat
    X_s^{(m)})| \; .
\end{align*}
As previously, we see that $\tilde W_{m,s}$ converges to 0 in $L^q$ for some
$q>\alpha$.  Thus, we obtain that
$$
\sum_{i=1}^n \delta_{(i/n,{\bf Y}_{n,i}^{(m)})} \Rightarrow
\sum_{s=0}^h \sum_{k=1}^{\infty} \delta_{(t_k,j_{k,s}^{(m)}{\bf
e}_s)} \; ,  (n\to+\infty) \; ,
$$
where $\sum_{k=1}^{\infty} \delta_{(t_k,j_{k,0}^{(m)})},
\dots,\sum_{k=1}^{\infty} \delta_{(t_k,j_{k,h}^{(m)})}$ are
independent Poisson processes with respective mean measures
$\lambda_{{s,m}}(dx)$, $s=0,\ldots,h$, defined in
(\ref{mean-measure-1})-(\ref{mean-measure-2}) with the constants
$d_+^{(m)}(s)$ and $d_-^{(m)}(s)$ that appear therein given by
\begin{align*}
  d_+^{(m)}(s) = d_+(s) \frac{\mathbb{E}[\exp(\alpha(X_0^{(m)}+\hat X_s^{(m)}))]}
  {\mathbb{E}[\exp(\alpha(X_0+\hat X_s))]} \; , \ \
  d_-^{(m)}(s) = d_-(s) \frac{\mathbb{E}[\exp(\alpha(X_0^{(m)}+\hat X_s^{(m)}))]}
  {\mathbb{E}[\exp(\alpha(X_0+\hat X_s))]} \; .
\end{align*}
Since $|\tilde W_{m,s}|$ converges to 0 in $L^q$, we obtain
$$
\sum_{k=1}^{\infty} \delta_{(t_k,j_{k,s}^{(m)})} \Rightarrow
\sum_{k=1}^{\infty} \delta_{(t_k,j_{k,s})} \; , (m\to+\infty) \; ,
\qquad s = 0,\dots,h \; .
$$
Then, for $s=1,\ldots,h$, we obtain, with $\tilde Z_0^{(s)} = Z_0
\exp(c_s
  \eta_0)$, for $\epsilon>0$,
\begin{align*}
  \limsup_{n\to+\infty} n \mathbb{P} \left(b_n^{-1} |Y_0Y_s - Y_0^{(m)}Y_s^{(m)}| > \epsilon \right)
  & = \limsup_{n\to+\infty} n \mathbb{P} \left(b_n^{-1} |Z_0\tilde Z_{0}^{(s)}| |\tilde
    W_{m,s}| > \epsilon \right) \leq C \epsilon^{-\alpha} \mathbb{E}[|\tilde W_{m,s}|^\alpha]
\end{align*}
which converges to 0 as $m\to+\infty$. This
proves~(\ref{eq:pp-tightness2}) and concludes the proof in the case
of leverage with~$\sigma(x) = \exp(x)$.

\subsubsection*{Case of leverage, $\sigma$ subadditive}
We have to bound
\begin{align*}
n\mathbb{P} (|Z_0Z_s||\sigma(X_0)\sigma(X_{s})-\sigma(X_0^{(m)})\sigma(X_{s}^{(m)})|>\epsilon b_n) \; .
\end{align*}
It suffices to bound two terms
\begin{align*}
  I_1(n,m)=n\mathbb{P} (|Z_0Z_s||\sigma(X_0)-\sigma(X_0^{(m)})|\sigma(X_{s}^{(m)})>\epsilon b_n) \; , \\
  I_2(n,m)=n\mathbb{P} (|Z_0Z_s|\sigma(X_0)|\sigma(X_{s})-\sigma(X_{s}^{(m)})|>\epsilon  b_n) \; .
\end{align*}
Recall that $X_s^{(m)}=\hat X_s^{(m)}+c_s\eta_0$ and $X_s=\hat
X_s+c_s\eta_0$. By subadditivity of $\sigma$, we have, for some constant $\delta$,
\begin{align*}
  I_1(n,m) \leq & n\mathbb{P} (|Z_0Z_s| |\sigma(X_0)-\sigma(X_0^{(m)})| \sigma(\hat
  X_{s}^{(m)}) > C \epsilon b_n) \\
  & + n\mathbb{P} (|Z_0Z_s| |\sigma(X_0)-\sigma(X_0^{(m)})| \sigma(c_s\eta_0) > \delta
  \epsilon b_n) \; .
\end{align*}
The product $Z_0Z_s$ is independent of
$|\sigma(X_0)-\sigma(X_0^{(m)})|\sigma(\hat X_{s}^{(m)})$ and tail equivalent to
$Y_0Y_1$, thus we obtain
$$
\limsup_{n\to+\infty}n\mathbb{P} (|Z_0Z_s||\sigma(X_0)-\sigma(X_0^{(m)})|\sigma(\hat
X_{s}^{(m)}) > \delta \epsilon b_n) \le C \epsilon^{-\alpha}
\mathbb{E}[|\sigma(X_0)-\sigma(X_0^{(m)})|^{\alpha} \sigma^{\alpha}(\hat X_{s}^{(m)})]
\; .
$$
We have already seen that $\sigma(X_0^{(m)})$ converges to $\sigma(X_0)$ in
$L^\alpha$, thus the latter expression converges to 0 as $m\to+\infty$.  By
assumption, $\sigma(c_s\eta_0)|Z_0Z_s|$ is either tail equivalent to $|Z_0Z_s|$
or $\mathbb{E}[\sigma^q(c_s\eta_0)|Z_0Z_s|^q]<\infty$ for some $q>\alpha$, and since
it is independent of $|\sigma(X_0) - \sigma(X_0^{(m)})|$, we obtain that
$$
\limsup_{n\to+\infty} n\mathbb{P} (\sigma(c_s\eta_0)|Z_0Z_s||\sigma(X_0) -
\sigma(X_0^{(m)})|>\epsilon b_n)  \leq C
\epsilon^{-\alpha}\mathbb{E}[|\sigma(X_0)-\sigma(X_0^{(m)})|^{\alpha}] \;,
$$
where $ C=0$ in the latter case. In both cases, this yields
\begin{align*}
  \lim_{m\to+\infty} \limsup_{n\to+\infty} n\mathbb{P} (\sigma(c_s\eta_0)|Z_0Z_s|
  |\sigma(X_0) - \sigma(X_0^{(m)})| > \epsilon b_n) = 0 \; .
\end{align*}
Thus we have obtained that $\lim_{m\to+\infty}\limsup_{n\to+\infty}I_1(n,m)=0$.

For the term $I_{2}(n,m)$ we use
assumption~(\ref{eq:sigma-condition-truncation}) with $x=c_s\eta_0$, $y=\hat
X_s$ and $z=\hat X_s^{(m)}$. Thus
\begin{align*}
  I_2(n,m) \leq n \mathbb{P} (|Z_0Z_s| (\sigma(c_s\eta_0) \vee 1) \tilde W_{m,s} >
  \epsilon b_n) \; ,
\end{align*}
with
\begin{align*}
  \tilde W_{m,s} = \sigma(X_0) \{(\sigma(\hat X_s) \vee 1) + (\sigma(\hat
  X_s^{(m)}) \vee 1)\}|\hat X_{s} - \hat X_{s}^{(m)}| \; .
\end{align*}
Note that $\tilde W_{m,s}$ is independent of $|Z_0Z_s| (\sigma(c_s\eta_0) \vee
1)$ and $\tilde W_{m,s}$ converges to 0 when $m\to+\infty$ in $L^q$ for some
$q>\alpha$. Since $|Z_0Z_s| \sigma(c_s\eta_0)$ is tail equivalent to $|Y_0Y_1|$
or has a finite moment of order $q'$ for some $q'>\alpha$, we have
\begin{align*}
  \limsup_{n\to+\infty} n \mathbb{P} (|Z_0Z_s| (\sigma(c_s\eta_0) \vee 1) \tilde W_{m,s}
  > \epsilon b_n) \leq C \mathbb{E}[\tilde W_{m,s}^\alpha] \; ,
\end{align*}
where the constant $C$ can be zero in the latter case. In both cases, we
conclude
\begin{align*}
  \lim_{m\to+\infty} \limsup_{n\to+\infty} n \mathbb{P} (|Z_0Z_s| (\sigma(c_s\eta_0) \vee
  1) \tilde W_{m,s} > \epsilon b_n) = 0 \; .
\end{align*}

\subsection{Proof of Theorem \ref{thm:partial-sums-egarch}}
We start by studying $S_{p,n}$.   Write
\begin{align*}
  \sum_{i=1}^{[nt]}\left(|Y_i|^p-\mathbb{E}[|Y_0|^p]\right) & =
  \sum_{i=1}^{[nt]}\left(|Y_i|^p-\mathbb{E}[|Y_i|^p|{\cal F}_{i-1}]\right)
  +\sum_{i=1}^{[nt]} \left(\mathbb{E}[|Y_i|^p|{\cal F}_{i-1}]-\mathbb{E}[|Y_0|^p]\right) \\
  & =: M_n(t)+R_n(t)\;.
\end{align*}
Note that $\mathbb{E}[|Y_i|^p|{\cal F}_{i-1}] = \mathbb{E}[|Z_0|^p]\sigma^p(X_i)$ is a
function of $X_i$ and does not depend on $Z_i$.  Then, by
\cite[Theorem~6]{arcones:1994}, for $\tau_p(1-H)<1/2$ we have
\begin{equation}
  \label{eq:lrd-limit-1-egarch}
  n^{-1} \rho_n^{-\tau_p/2} R_n \stackrel{\scriptstyle \mathcal D}{\Rightarrow} \frac{J_{\tau_p}(\sigma^p) \mathbb{E}[|Z_1|^p]}{\tau_p !} R_{\tau_p,H} \; .
\end{equation}
If $\tau_p(1-H)>1/2$ then by \cite[Theorem~4]{arcones:1994}, we obtain
\begin{equation}
  \label{eq:lrd-limit-2-egarch}
  n^{-1/2} R_n \stackrel{\scriptstyle \mathcal D}{\Rightarrow} \varsigma \mathbb{E}[|Z_0|^p] B  \; ,
\end{equation}
where $B$ is the standard Brownian motion and $\varsigma^2 = \var(\sigma^p(X_0))+
2\sum_{i=1}^\infty \mathrm{cov}(\sigma^p(X_0),\sigma^p(X_i))$.  We will show that
under the assumptions of Theorem \ref{thm:partial-sums-egarch} we have,
\begin{eqnarray}
  \label{eq:Levy-conv-egarch}
  a_n^{-p} M_n \stackrel{\scriptstyle \mathcal D}{\Rightarrow} L_{\alpha/p}\; .
\end{eqnarray}
The convergences~(\ref{eq:lrd-limit-1-egarch}), (\ref{eq:lrd-limit-2-egarch})
and (\ref{eq:Levy-conv-egarch}) conclude the proof of the theorem.  We now
prove~(\ref{eq:Levy-conv-egarch}). The proof is very similar to the proof of the
convergence of the partial sum of an i.i.d.~sequence in the domain of attraction
of a stable law to a L\'evy stable process. The differences are some additional
technicalities. See e.g. \cite[Proof of Theorem~7.1]{resnick:2007} for more
details.  For $0<\epsilon<1$, decompose it further as
\begin{align*}
  M_n(t) & = \sum_{i=1}^{[nt] }\left\{ |Y_i|^p \mathbf1_{\{|Y_i|<\epsilon
      a_n\}}-\mathbb{E}\left[|Y_i|^p\mathbf1_{\{|Y_i|<\epsilon a_n\}}|{\cal F}_{i-1}\right]\right\} \\
  & + \sum_{i=1}^{[nt] } \left\{ |Y_i|^p \mathbf 1_{\{|Y_i| > \epsilon a_n\}} -
    \mathbb{E}\left[|Y_i|^p\mathbf1_{\{|Y_i|>\epsilon a_n\}}|{\cal F}_{i-1}\right]\right\} =:
  M_n^{(\epsilon)}(t)+\tilde M_n^{(\epsilon)}(t) \; .
\end{align*}
The term $\tilde M_n^{(\epsilon)}(\cdot)$ is treated using the point process
convergence. Since for any $\epsilon>0$, the summation functional is almost
surely continuous from the set of Radon measures on $[0,1] \times
[\epsilon,\infty)$ onto $\mathcal D([0,1],\mathbb R)$ with respect to the
distribution of the Poisson point process with mean measure $\nu_0$ (see
e.g. \cite[p.~215]{resnick:2007}), from Proposition \ref{prop:pp-univarie} we
conclude
\begin{align}
  \label{eq:proof-1b-egarch}
  a_n^{-p} \sum_{i=1}^{[n\cdot] } |Y_i|^p \mathbf 1_{\{|Y_i|>\epsilon a_n\}}
  \stackrel{\scriptstyle \mathcal D}{\Rightarrow} \sum_{t_k\le (\cdot)} |j_k|^p \mathbf 1_{\{|j_k|>\epsilon\}} \; .
\end{align}
Taking expectation in (\ref{eq:proof-1b-egarch}) we obtain
\begin{align}
  \label{eq:proof-1c-egarch}
  \lim_{n\to+\infty} [nt] a_n^{-p} \mathbb{E} \left[|Y_0|^p \mathbf
    1_{\{|Y_1|>\epsilon a_n\}} \right] = t \int_{\{x:|x|>\epsilon\}} |x|^p
  \lambda_0(\mathrm d x)
\end{align}
uniformly with respect to $t\in[0,1]$ since it is a sequence of increasing
functions with a continuous limit. Furthermore, we claim that
\begin{align}
  \label{eq:proof-1d-egarch}
  a_n^{-p} \left|\sum_{i=1}^{[nt] } \left\{ \mathbb{E}\left[|Y_0|^p \mathbf
        1_{\{|Y_1|>\epsilon a_n\}} \right] - \mathbb{E} \left[ |Y_i|^p \mathbf
        1_{\{|Y_i|>\epsilon a_n\}}|{\cal F}_{i-1}\right] \right\} \right|
  \stackrel{\scriptstyle P}{\to} 0\;,
\end{align}
uniformly in $t\in [0,1]$. We use the variance inequality
(\ref{eq:variance-inequality-lrd}) to bound the variance of the last expression
by
\begin{align*}
  a_n^{-2p} [nt]^2 \rho_{[nt]} \,
  \mathrm{var}\left(\mathbb{E}[|Y_1|^p\mathbf1_{\{|Y_1|>\epsilon a_n\}}|{\cal F}_0]\right) \leq
  a_n^{-2p} [nt]^2 \rho_{[nt]} \mathbb{E} \left[ \left(\mathbb{E}[|Y_1|^p \mathbf
      1_{\{|Y_1|>\epsilon a_n\}}|{\cal F}_0]\right)^2 \right] \; .
\end{align*}
If $p<\alpha<2p$, by Karamata's Theorem (see \cite[p.~25]{resnick:2007}) and
Potter's bound,
$$
\mathbb{E} [\sigma^p(x) |Z_1|^p \mathbf 1_{\{|\sigma(x)Z_1|>\epsilon a_n\}}] \leq C
n^{-1} a_n^{p} \frac{\bar F_Z(\epsilon a_n/\sigma(x))}{\bar F_Z(a_n)} \leq C n^{-1}
a_n^p \sigma^{\alpha+\epsilon} (x)\;.
$$
Since by assumption $\mathbb{E}[\sigma^{2\alpha+2\epsilon}(X_0)] < \infty$ for some
$\epsilon>0$, for each $t$, we have
\begin{multline}
  \mathrm{var} \left( a_n^{-p} \sum_{i=1}^{[nt] } \left\{ \mathbb{E}\left[|Y_0|^p \mathbf
        1_{\{|Y_0|>\epsilon a_n\}} \right] - \mathbb{E} \left[ |Y_i|^p \mathbf
        1_{\{|Y_i|>\epsilon a_n\}}|{\cal F}_{i-1}\right] \right\} \right) \\
  \leq C n^{-2} [nt]^2 \rho_{[nt]} \leq C n^{2H-2+\epsilon} t^{2H-\epsilon} \;  , \label{eq:same-arguments}
\end{multline}
where the last bound is obtained for some $\epsilon>0$ by Potter's bound.  This
proves convergence of finite dimensional distribution to 0 and tightness in
$\mathcal D([0,1],\mathbb R)$.
As in \cite[p.~216]{resnick:2007}, we now argue that (\ref{eq:proof-1b-egarch}),
(\ref{eq:proof-1c-egarch}) and~(\ref{eq:proof-1d-egarch}) imply that
\begin{align}
  \label{eq:weak-conv-egarch}
  a_n^{-p} \tilde M_n^{(\epsilon)} \stackrel{\scriptstyle \mathcal D}{\Rightarrow} L_{\alpha/p}^{(\epsilon)} \; ,
\end{align}
and it also holds that $L_{\alpha/p}^{(\epsilon)} \stackrel{\scriptstyle \mathcal D}{\Rightarrow} L_{\alpha/p}$ as
$\epsilon\to0$.
Therefore, to show (\ref{eq:Levy-conv-egarch}) is suffices to show the
negligibility of $a_n^{-p}M_n^{(\epsilon)}$. By Doob's martingale inequality we
evaluate
\begin{align*}
  \mathbb{E} &\left[ \left(\sup_{t\in [0,1]} a_n^{-p} \sum_{i=1}^{[nt]} \left\{
        |Y_i|^p \mathbf1_{\{|Y_i|<\epsilon a_n\}} - \mathbb{E} \left[|Y_i|^p \mathbf1_{\{|Y_i|< \epsilon a_n\}}| {\cal F}_{i-1} \right] \right\} \right)^2 \right]  \\
  & \leq C na_n^{-2p} \mathbb{E} \left[ \left(|Y_1|^p \mathbf1_{\{|Y_1|<\epsilon
        a_n\}} - \mathbb{E} \left[|Y_1|^p \mathbf1_{\{|Y_1|<\epsilon a_n\}}|{\cal
          F}_{0}\right]\right)^2  \right] \\
  &\leq 4 C n a_n^{-2p} \mathbb{E} \left[|Y_1|^{2p} \mathbf1_{\{|Y_1|<\epsilon a_n\}}  \right] \;.
\end{align*}
Recall that $\alpha<2p$. By Karamata's theorem (see \cite[p. 25]{resnick:2007}),
\begin{align}
  \label{eq:Karamata}
  \mathbb{E} \left[ |Y_1|^{2p} \mathbf1_{\{|Y_1|<\epsilon a_n\}}\right] & \sim
  \frac{2\alpha}{2p-\alpha}(\epsilon a_n)^{2p} \bar F_Y(\epsilon a_n) \sim
  \frac{2\alpha}{2p-\alpha}\epsilon^{2p-\alpha}a_n^{2p}n^{-1} \; .
\end{align}
Applying this and letting $\epsilon\to 0$ we conclude that
$a_n^{-p}M_n^{(\epsilon)}$ is uniformly negligible in $L^2$ and so in
probability, and thus we conclude that $a_n^{-p}M_n\stackrel{\scriptstyle \mathcal D}{\Rightarrow} L_{\alpha/p}$.

For $p>\alpha$, $\mathbb{E}[|Y_0|^p]=\infty$. In that case it is well known (see
e.g. \cite[Theorem 3.1]{davis:hsing:1995}) that the convergence of $a_n^{-p}
S_{p,n}$ to an $\alpha/p$-stable L\'evy process follows directly from the
convergence of the point process $\sum_{i=1}^n \delta_{Y_i/a_n}$ to a Poisson
point process, and that no centering is needed. In the present context, this
entirely dispenses with the conditioning argument and the long memory part does
not appear. Therefore convergence to stable L\'evy process always holds.

As for the sum $S_n$, since $\mathbb{E}[Y_0]=\mathbb{E}[Z_0]=0$, the long memory part $R_n$
is identically vanishing, thus in this case also only the stable limit arises.

\subsection{Proof of Theorem~\ref{theo:cov-lmsv}}
Let $U_i = |Y_iY_{i+s}|$.  We now write
\begin{align*}
  \sum_{i=1}^{n} \left(U_i^p - \mathbb{E}[U_0^p] \right) & = \sum_{i=1}^{n} \left(U_i^p -
    \mathbb{E}[U_i^p \mid {\cal F}_{i-1}] \right) + \sum_{i=1}^{n}
  \left(\mathbb{E}[U_i^p \mid {\cal F}_{i-1}] - \mathbb{E}[U_0^p] \right)  \\
  & = M_{n,s} + \sum_{i=1}^{n} K_p^*(X_i,\hat X_{i,s}) = M_{n,s} + T_{n,s}\;.
\end{align*}
As mentioned above, the second part is the partial sum of a sequence
of a function of the bivariate Gaussian sequence $(X_i,\hat
X_{i,s})$. The proof of the convergence to a stable law mimics the
proof of Theorem \ref{thm:partial-sums-egarch}. We split $M_{n,s}$
between big jumps and small jumps. Write $M_{n,s}^{(\epsilon)} +
\tilde M_{n,s}^{(\epsilon)}$, with
\begin{align*}
  M_{n,s}^{(\epsilon)} = \sum_{i=1}^{n} \left(U_i^p \mathbf 1_{\{U_i \leq b_n
      \epsilon\}} - \mathbb{E}[U_i^p \mathbf 1_{\{U_i \leq b_n \epsilon\}} \mid {\cal
      F}_{i-1}] \right) \; .
\end{align*}
The point process convergence yields the convergence of the big jumps parts by
the same argument as in the proof of Theorem~\ref{thm:partial-sums-egarch}.  In
order to prove the asymptotic negligibility of the small jumps parts, the only
change that has to be made comes from the observation that $\tilde
M_{n,s}^{(\epsilon)}$ is no longer a martingale. However, assuming for
simplicity that we have $(s+1)n$ observations $Y_i$, we write, with $U_{i,k} =
U_{(s+1)i-k} = |Y_{(s+1)i-k}Y_{(s+1)i+s-k}|$,
$$
M_{n,s}^{(\epsilon)} = \sum_{k=0}^s \sum_{i=1}^{n} \left\{U_{i,k}^p \mathbf
  1_{\{U_i \leq b_n \epsilon\}} - \mathbb{E}\left[U_{i,k}^p \mathbf 1_{\{U_i \leq b_n
      \epsilon\}} \mid {\cal F}_{(s+1)i-k-1}\right] \right\} = : \sum_{k=0}^s
M_{n,s,k}^{(\epsilon)} \; .
$$
Clearly, each $M_{n,s,k}^{(\epsilon)}$, $k=0,\ldots,s$, is a
martingale with respect to the filtration $\{\mathcal F_{i(s+1)}, 1
\leq i \leq n\}$, therefore we can apply Doob's inequality and
conclude the proof with the same arguments as previously.

\subsection{Proof of Theorem~\ref{theo:cov-egarch}}
Again, we mimic the proof of Theorem \ref{thm:partial-sums-egarch}, however,
some technical modifications are needed.  We use the decomposition between small
jumps and big jumps. To prove negligibility of the small jumps, we use the same
splitting technique as in the proof of Theorem~\ref{theo:cov-lmsv}. To deal with
the big jumps, the only adaptation needed is to obtain a bound for the quantity
\begin{align}
  b_n^{-2p} n^2 \rho_{n} \mathbb{E} \left[ \left(\mathbb{E}[|Y_0Y_{s}|^p \mathbf
      1_{\{|Y_0Y_{s}| > \epsilon b_n\}}|{\cal F}_{-1}] \right)^2 \right] \;  . \label{eq:variance-a-borner}
\end{align}
To show that (\ref{eq:proof-1d-egarch}) still holds in the present context, we
must prove that the expectation in~(\ref{eq:variance-a-borner}) is of order
$n^{-2} b_n^{2p}$. The rest of the arguments to prove the convergence of the
big jumps part remains unchanged.
Note that $\mathbb{E}[|Y_0Y_{s}|^p \mathbf 1_{\{|Y_0Y_{s}| > \epsilon b_n\}}|{\cal
  F}_{-1}] = G(X_0,\hat X_{0,s})$, thus we need an estimate for the bivariate
function
\begin{align*}
  G(x,y) = \sigma^p(x)\mathbb{E} [|Z_0Z_s|^p \sigma^p(c_s\eta_0 + \varsigma_s \zeta+y)
  \mathbf1_{\{|Z_0Z_s| \sigma(c_s\eta_0 + \varsigma_s\zeta+y) > \epsilon b_n\}}] \; ,
\end{align*}
where $\zeta$ is a standard Gaussian random variable, independent of
$Z_0,\eta_0$ and $Z_s$. We obtain this estimate first in the case
$\sigma(x)=\exp(x)$ and then for subadditive functions.

Let $\sigma(x)=\exp(x)$.  As in the proof of point process
convergence, we write
$$Y_0Y_{s}=Z_0Z_{s}\exp(c_s\eta_0)\exp\left(X_0+\hat X_{s}\right).$$
By Lemma~\ref{lemma:asymp-indep-EGARCH-expo},
$Z_0Z_{s}\exp(c_s\eta_0)$ is regularly varying and tail equivalent
to $Z_0Z_{s}$. Since $ \exp(p\varsigma_s\zeta)$ is independent of
$Z_0Z_s\exp(c_s\eta_0)$ and has finite moments of all orders, we
obtain that $Z_0Z_{s}\exp(c_s\eta_0)\exp(p\varsigma_s\zeta)$ is also
tail equivalent to $Z_0Z_s$, hence to $Y_0Y_1$. Thus, by Karamata's
Theorem and Potter's bounds,  we obtain, for some $\delta>0$,
\begin{align*}
  G(x,y) &= \exp(p(x+y)) \mathbb{E} [|Z_0Z_s|^p \exp(pc_s \eta_0) \exp(p\varsigma_s\zeta) \mathbf1_{\{|Z_0Z_s|
  \exp(pc_s \eta_0)
\exp(\varsigma_s\zeta)) > \epsilon b_n \exp(-y)\}}]\\
& \leq C n^{-1} b_n^p
  \exp(px) \exp((p-\alpha+\delta)(y\vee0)) \; .
\end{align*}
Since the log-normal distribution has finite moments of all order,
we obtain that $\mathbb{E}[G^2(X_0,\hat{X}_{0,s})] = O(n^{-2}b_n^{2p})$ which is
the required bound. This concludes the proof in the case $\sigma(x)
=\exp(x)$.

Let now  the assumptions of Proposition~\ref{prop:egarch-pp-quadratic} be
in force.  Using the subadditivity of $\sigma^p$, we obtain $G(x,y) \leq
\sum_{i=1}^4 I_i(x,y)$ with
\begin{align*}
  I_1(x,y) &  = \sigma^p(x)\mathbb{E}[|Z_0Z_s|^p \sigma^p(\vartheta_s) \mathbf 1_{\{|Z_0Z_s| \sigma (\vartheta_s) > \epsilon b_n\}}],\\
  I_2(x,y )& = \sigma^p(x)\mathbb{E}[|Z_0Z_s|^p\sigma^p(y) \mathbf 1_{\{|Z_0Z_s|\sigma(y) > \epsilon b_n\}}],\\
  I_3(x,y) &  = \sigma^p(x)\mathbb{E}[|Z_0Z_s|^p\sigma^p(\vartheta_s) \mathbf 1_{\{|Z_0Z_s|\sigma(y)>\epsilon b_n\}}],\\
  I_4(x,y) & = \sigma^p(x)\mathbb{E}[|Z_0Z_s|^p\sigma^p(y) \mathbf  1_{\{|Z_0Z_s|\sigma(\vartheta_s)>\epsilon b_n\}}] \; ,
\end{align*}
where for brevity we have denoted $\vartheta_s = c_s\eta_0+
\varsigma_s\zeta$. We now give bound $\mathbb{E}[I_j^2(X_0,\hat{X}_{0,s})]$,
$j=1,2,3,4$.  Since by the assumptions, $|Z_0Z_s|\sigma(\vartheta_s)$
is tail equivalent to $|Z_0Z_s|$, Karamata's Theorem yields
\begin{align*}
  \sigma^p(x) \mathbb{E} [|Z_0Z_s|^p \sigma^p(\vartheta_s) \mathbf
  1_{\{|Z_0Z_s|\sigma(\vartheta_s)>\epsilon b_n\}}] \le Cn^{-1}b_n^{p}\sigma^p(x) \; ,
\end{align*}
and since $\mathbb{E}[\sigma^{2p}(X_0)]<\infty$ by assumption, we obtain
by integrating that $\mathbb{E}[I_1^2(X_0,\hat{X}_{0,s})] = O(n^{-2} b_n^{2p})$.
For $I_2$, using again Karamata's Theorem and Potter's bound, we
obtain, for some $\delta>0$,
\begin{align*}
  \sigma^p(x) \mathbb{E} [|Z_0Z_s|^p \sigma^p(y) \mathbf1_{\{|Z_0Z_s|\sigma(y)>\epsilon
    b_n\}}] \le Cn^{-1} b_n^{p} \sigma^p(x) (\sigma(y)\vee1)^{p-\alpha+\delta} \; .
\end{align*}
Since $|Z_0|\sigma(\vartheta_s)$ is tail equivalent to $|Z_0|$ and $Z_s$ is
independent of $Z_0\sigma(\vartheta_s)$, we easily obtain a bound for the tail
of $|Z_0Z_s|(\sigma(\vartheta_s)\vee 1)$:
$$
\mathbb{P} (|Z_0Z_s|(\sigma(\vartheta_s)\vee 1)>x) \le
\mathbb{P}(|Z_0Z_s|\sigma(\vartheta_s)>x) + \mathbb{P}(|Z_0Z_s|>x) \leq C
\mathbb{P}(Z_0Z_s>x) \; ,
$$
for $x$ large. Thus, applying Karamata's Theorem and Potter's bound
to $|Z_0Z_s|$ yields, for some arbitrarily small $\delta>0$,
\begin{align*}
  I_3(x,y) \leq C \sigma^p(x) \mathbb{E}[|Z_0Z_s|^p \mathbf 1_{\{\sigma(y)|Z_0Z_s| > \epsilon
    b_n\}}] \leq C n^{-1} b_n^p \sigma^p(x) (\sigma(y) \vee 1)^{\alpha+\delta} \;
\end{align*}
and thus we conclude that $\mathbb{E}[I_3^2(X_0,\hat{X}_{0,s})] = O(n^{-2}b_n^{2p})$.
Finally, we write,
$$
I_4(x,y) \le \sigma^p(x)\sigma^p(y) \mathbb{E}[|Z_0Z_s|^p\left(\sigma^p(\vartheta_s)\vee
  1\right) \mathbf 1_{\{|Z_0Z_s|\left(\sigma(\vartheta_s)\vee 1\right)>\epsilon b_n\}}]
$$
and by the same argument as for $I_3$ we obtain that
$\mathbb{E}[I_3^2(X_0,\hat{X}_{0,s})] = O(n^{-2}b_n^{2p})$.

\appendix

\section{Gaussian long memory sequences}
\label{sec:LRD-Gaussian}

For the sake of completeness, we recall in this appendix the main definitions
and results pertaining to Hermite coefficients and expansions of square
integrable functions with respect to a possibly non standard
multivariate Gaussian distribution. Expansions with respect to the multivariate
standard Gaussian distribution are easy to obtain and describe. The theory for
non standard Gaussian vectors is more cumbersome.  The main reference
is~\cite{arcones:1994}.

\subsection{Hermite coefficients and rank}
Let $G$ be a function defined on $\mathbb{R}^k$ and $\mathbf
X=(X^{(1)},\ldots,X^{(k)})$ be a $k$-dimensional centered Gaussian vector with
covariance matrix $\Gamma$. The Hermite coefficients of $G$ with respect to $\mathbf{X}$
are defined as
\begin{align*}
  J(G,\mathbf X,\mathbf q) = \mathbb{E}\left[G(\mathbf X)\prod_{j=1}^k H_{q_j}
    (X^{(j)})\right] \; ,
\end{align*}
where $\mathbf q = (q_1,\ldots,q_k) \in\mathbb N^k$. If $\Gamma$ is the $k\times
k$ identity matrix (denoted by $I_k$), i.e. the components of $\mathbf X$ are
i.i.d.~standard Gaussian, then the corresponding Hermite coefficients are
denoted by $J^*(G,\mathbf q)$. The Hermite rank of $G$ with respect to $\mathbf
X$, is the smallest integer $\tau$ such that
\begin{align*}
  J(G,\mathbf X,\mathbf q) = 0 \ \mbox{ for all } \ \mathbf q \mbox{ such that }
  \ 0<| q_1+\cdots+q_k| < \tau \; .
\end{align*}

\subsection{Variance inequalities}
Consider now a $k$-dimensional stationary centered Gaussian process $\{\mathbf
X_i,i\ge 0\}$ with covariance function $\rho_n(i,j) = \mathbb{E}[ X_0^{(i)}
X_n^{(j)}]$ and assume either
\begin{align}
  \label{eq:weak-dependence} \forall 1 \leq i,j \leq k \;, \ \ \sum_{n=0}^\infty
  |\rho_n(i,j)| < \infty \; ,
\end{align}
or that there exists $H\in(1/2,1)$ and a function $\ell$ slowly
varying at infinity such that
\begin{align}
  \label{eq:lrd}
  \lim_{n\to+\infty} \frac{\rho_n(i,j)}{n^{2H-2}\ell(n)} = b_{i,j} \; ,
\end{align}
and the $b_{i,j}$s are not identically zero.  Denote then $\rho_n =
n^{2H-2}\ell(n)$. Then, we have the following cases.
\begin{itemize}
\item If~(\ref{eq:lrd}) holds and $2\tau(1-H)<1$, then for any function $G$ with
  Hermite rank $\tau$ with respect to $\mathbf X_0$,
\begin{align}
  \label{eq:variance-inequality-lrd}
  \mathrm{var}\left( n^{-1}\sum_{j=1}^n G(\mathbf X_j) \right) \leq C \rho_n^\tau \;
  \mathrm{var}(G(\mathbf X_0)) \; .
\end{align}
\item If~(\ref{eq:lrd}) holds and $2\tau(1-H)>1$, then for any function $G$ with
  Hermite rank $\tau$ with respect to $\mathbf X_0$,
\begin{align}
  \label{eq:variance-inequality-weak-dependence}
  \mathrm{var}\left( \sum_{j=1}^n G(\mathbf X_j) \right) \leq C n \; \mathrm{var}(G(\mathbf X_0))
  \; .
\end{align}
\item If~(\ref{eq:weak-dependence}) holds,
  then~(\ref{eq:variance-inequality-weak-dependence}) still holds.
\end{itemize}
In all these cases, the constant $C$ depends only on the Gaussian process
$\{\mathbf X_i\}$ and not on the function $G$. The
bounds~(\ref{eq:variance-inequality-lrd})
and~(\ref{eq:variance-inequality-weak-dependence}) are Equation~3.10 and~2.40 in
\cite{arcones:1994}, respectively. The
bound~(\ref{eq:variance-inequality-weak-dependence}) under
assumption~(\ref{eq:weak-dependence}) is a consequence of Equation~2.18 in
\cite[Theorem~2]{arcones:1994}.

\subsection{Limit theorems}
\label{sec:non-clt}
We now recall \cite[Theorem~6]{arcones:1994}.  Let again $\{\mathbf X_i\}$ be a
stationary sequence of $k$-dimensional Gaussian vectors with covariance matrix
$G$ and such that~(\ref{eq:lrd}) holds, and let $\tau$ be the Hermite rank of
$G$ w.r.t. $\mathbf X_0$.  If $\tau(1-H)<1/2$, there exists a process
$R_{G,\tau,H}$ such that
\begin{align}
  \label{eq:LRD-conv-multivarie}
  \frac{1}{n \rho_n^{\tau/2}} \sum_{i=1}^{[n\cdot]} \left(G( {\mathbf
      X}_i)-\mathbb{E}\left[G( {\mathbf X}_0)\right]\right) \stackrel{\scriptstyle \mathcal D}{\Rightarrow} R_{G,\tau,H}  \; .
\end{align}
In particular, if $k=1$, then
\begin{align}
  \label{eq:lim-sums}
  \frac{1}{n\rho_n^{\tau/2}} \sum_{i=1}^{[n\cdot]} \{G({X}_i) -
  \mathbb{E}[G(X_0)]\} \stackrel{\scriptstyle \mathcal D}{\Rightarrow}
  \frac{J_\tau(G)}{\tau!} R_{\tau,H} \; ,
\end{align}
where $J_\tau(G)=\mathbb{E}[G(X_1)H_{\tau}(X_1)]$ and $R_{\tau,H}$ is the
so-called Hermite or Rosenblatt process of order $\tau$, defined as a
$\tau$-fold stochastic integral
\begin{align*}
  R_{\tau,H}(t) =K_1(\tau,H)\int_{-\infty}^\infty \dots \int_{-\infty}^\infty
  \frac{\exp\left(\mathrm i t(x_1+\cdots+x_{\tau})\right)-1}{x_1+\cdots+x_{\tau}} \,
  \prod_{i=1}^{\tau} x_i^{-H+1/2} \, W(\mathrm{d}x_1) \dots W(\mathrm{d}x_{\tau}) \; ,
\end{align*}
where $W$ is an independently scattered Gaussian random measure with Lebesgue
control measure and
$$
K_1^2(\tau,H) = \frac{(\tau(H-1)+1)(2\tau(H-1)+1)}
{\tau!\left\{2\Gamma(2-2H) \sin \pi(H-\frac{1}{2})\right\}^\tau} \; .
$$
In particular, for $\tau=1$, then the limiting process is the fractional
Brownian motion, which is a Gaussian process, so
\begin{align*}
  \frac{1}{n\rho_n^{1/2}} \sum_{i=1}^{n} \{G({X}_i) - \mathbb{E}[G(X_0)]\}
  \stackrel{\scriptstyle d}{\to} \mathbf N \left(0, \frac{J(1)}{H(2H-1)} \right)  \; .
\end{align*}
On the other hand, if $1-\tau(1-H)<1/2$, then
\begin{equation}
  \label{eq:lim-sums-iid}
  \frac{1}{\sqrt{n}} \sum_{i=1}^{[n\cdot]} \{G({X}_i) - \mathbb{E}[G(X_0)]\} \stackrel{\scriptstyle \mathcal D}{\Rightarrow}  \varsigma  B \; ,
\end{equation}
where $B$ is the standard Brownian motion and $\varsigma^2 = \mathrm{var}(G(X_0)) +
2\sum_{j=1}^\infty \mathrm{cov}(G(X_0),G(X_j))$, the latter series being
absolutely summable.

{\bf Acknowledgements:} We are grateful to an anonymous referee whose remarks
lead to a substantial improvement of our article. The research of the first
author was supported by the NSERC grant. The research of the second author was
partially supported by the ANR grant ANR-08-BLAN-0314-02.

\end{document}